\documentclass[nohyperref]{article}

\usepackage{microtype}
\usepackage{graphicx}
\usepackage{subfigure}
\usepackage{booktabs} 

\usepackage{hyperref}

\usepackage[accepted]{icml2022}

\usepackage{amsmath}
\usepackage{amssymb}
\usepackage{mathtools}
\usepackage{amsthm}

\usepackage[capitalize,noabbrev]{cleveref}

\theoremstyle{plain}
\newtheorem{theorem}{Theorem}[section]

\newtheorem{lemma}[theorem]{Lemma}

\theoremstyle{definition}
\newtheorem{definition}[theorem]{Definition}
\newtheorem{assumption}[theorem]{Assumption}
\theoremstyle{remark}

\usepackage[textsize=tiny]{todonotes}

\icmltitlerunning{Accelerated Primal-Dual Gradient Method for Smooth and Convex-Concave Saddle-Point Problems with Bilinear Coupling}

\usepackage{colortbl}
\definecolor{bgcolor}{rgb}{0.50,0.90,0.50}

\usepackage{nicefrac}
\usepackage{makecell}
\usepackage{nicematrix}

\DeclareMathOperator{\dom}{\mathrm{dom}}

\newcommand{\range}{\mathrm{range}}

\newcommand{\Ec}[2]{\mathbb{E}\left[#1\;\middle|\;#2\right]}

\newcommand{\lmax}{\lambda_{\max}}
\newcommand{\lmin}{\lambda_{\min}}
\newcommand{\lminp}{\lambda_{\min}^+}

\newcommand{\bg}{\mathrm{B}}

\newcommand{\R}{\mathbb{R}}

\def\<#1,#2>{\langle #1,#2\rangle}

\newcommand{\norm}[1]{\|#1\|}
\newcommand{\sqn}[1]{\norm{#1}^2}

\newcommand{\cS}{\mathcal{S}}

\newcommand{\cO}{\mathcal{O}}

\newcommand{\mA}{\mathbf{A}}
\newcommand{\mB}{\mathbf{B}}
\newcommand{\mC}{\mathbf{C}}

\newcommand{\mW}{\mathbf{W}}

\newcommand{\mI}{\mathbf{I}}

\begin{document}

\twocolumn[
\icmltitle{Accelerated Primal-Dual Gradient Method for Smooth and Convex-Concave Saddle-Point Problems with Bilinear Coupling}



\icmlsetsymbol{equal}{*}

\begin{icmlauthorlist}
\icmlauthor{Dmitry Kovalev}{kaust}
\icmlauthor{Alexander Gasnikov}{mipt}
\icmlauthor{Peter Richt\'{a}rik}{kaust}
\end{icmlauthorlist}

\icmlaffiliation{kaust}{King Abdullah University of Science and Technology, Thuwal, Saudi Arabia}
\icmlaffiliation{mipt}{Moscow Institute of Physics and Technology, Dolgoprudny, Russia}

\icmlcorrespondingauthor{Dmitry Kovalev}{dakovalev1@gmail.com}

\icmlkeywords{}

\vskip 0.3in
]



\printAffiliationsAndNotice{}  

\begin{abstract}
	In this paper we study the convex-concave saddle-point problem $\min_x \max_y f(x) + y^T \mA x - g(y)$, where $f(x)$ and $g(y)$ are smooth and convex functions. We propose an Accelerated Primal-Dual Gradient Method (APDG) for solving this problem, achieving (i) an optimal linear convergence rate in the strongly-convex-strongly-concave regime, matching the lower complexity bound (Zhang et al., 2021), and (ii) an accelerated linear convergence rate in the case when only one of the functions $f(x)$ and $g(y)$ is strongly convex or even none of them are. Finally, we obtain a linearly convergent algorithm for the general smooth and convex-concave saddle point problem $\min_x \max_y F(x,y)$ without the requirement of strong convexity or strong concavity.
\end{abstract}

\section{Introduction}

In this paper we revisit the well studied smooth convex-concave saddle point problem with a bilinear coupling function, which takes the form
\begin{equation}\label{eq:main}
	\min_{x \in \R^{d_x}}\max_{y \in \R^{d_y}} F(x,y) = f(x) + y^\top\mA x - g(y),
\end{equation}
where $f(x)\colon\R^{d_x} \rightarrow \R$ and $g(y)\colon\R^{d_y} \rightarrow \R$ are smooth and convex functions, and $\mA \in \R^{d_y \times d_x}$ is a coupling matrix.

Problem \eqref{eq:main} has a large number of application, some of which we now briefly introduce.

\subsection{Empirical risk minimization} A classical application is the  regularized empirical risk minimization (ERM)  with linear predictors, which is a classical supervised learning problem. Given a data matrix $\mA = [a_1,\ldots,a_n]^\top \in \R^{n\times d}$, where $a_i\in \R^d$ is the feature vector of the $i$-th data entry, our goal is to find a solution of
\begin{equation}\label{eq:erm}
	\min_{x} f(x) + \ell(\mA x),
\end{equation}
where $f(x):\R^d \to \R$ is a convex regularizer, $\ell(y):\R^n \to \R$ is a convex loss function, and $x \in \R^d$ is a linear predictor.
Alterantively, one can solve the following equivalent saddle-point reformulation of problem \eqref{eq:erm}:
\begin{equation}
	\min_{x}\max_y  f(x) + y^\top \mA x - \ell^*(y).
\end{equation}
The saddle-point reformulation is often preferable. For example, when such a formulation admits a finite sum structure \citep{zhang2015stochastic, wang2017exploiting}, this may reduce the communication complexity in the distributed setting \citep{xiao2019dscovr}, and one may also better exploit the udnerlying sparsity structure \citep{lei2017doubly}.

\subsection{Reinforcement learning} In reinforcement learning (RL)  we are given a sequence $\{(s_t,a_t, r_t,s_{t+1})\}_{t=1}^n$ generated by a policy $\pi$, where $s_t$ is the state at time step $t$, $a_t$ is the action taken at time step $t$ by policy $\pi$ and $r_t$ is the reward after taking action $a_t$. A key step in many RL algorithms is to estimate the value function of a given policy $\pi$, which is defined as
\begin{equation}
	V^\pi(s) = \Ec{\sum_{t=0}^\infty \gamma^t r_t}{s_0 = s, \pi},
\end{equation}
where $\gamma \in (0,1)$ is a discount factor. A common approach to this problem is to use a linear approximation $V^\pi(s) = \phi(s)^\top x$, where $\phi(s)$ is a feature vector of a state $s$. The model parameter $x$ is often estimated by minimizing the mean squared projected Bellman error
\begin{equation}\label{eq:mspbe}
	\min_x \sqn{\mB x - b}_{\mC^{-1} },
\end{equation}
where $\mC  = \sum_{t=1}^n \phi(s_t)\phi(s_t)^\top$, $b= \sum_{t=1}^n r_t \phi(s_t)$ and $\mB = \mC - \gamma \sum_{t=1}^n \phi(s_t)\phi(s_{t+1})^\top$. One can observe that it is hard to apply gradient-based methods to problem \eqref{eq:mspbe} because this would require one to compute an inverse of the matrix $\mC$. In order to tackle this issue, one can solve an equivalent saddle-point reformulation proposed by \citet{du2017stochastic} instead. This reformulation is given by 
\begin{equation}
	\min_x\max_y   - 2y^\top \mB x - \sqn{y}_\mC + 2b^\top y,
\end{equation}
and is an instance of problem~\eqref{eq:main}.
Solving this reformulation with gradient methods does not require matrix inversion.

\subsection{Minimization under affine constraints}

Next,   consider the problem of convex minimization under affine constraints,
\begin{equation}\label{eq:lc}
	\min_{\mA x = b} f(x), 
\end{equation}
where $b \in \range \mA$. 
This problem covers a wide range of applications, including inverse problems in imaging \citep{chambolle2016introduction}, sketched learning-type applications \citep{keriven2018sketching}, network flow optimization \citep{zargham2013accelerated} and optimal transport \citep{peyre2019computational}.

Another important application of problem~\eqref{eq:lc} is decentralized distributed optimization
\citep{kovalev2020optimal,scaman2017optimal,li2020optimal,nedic2017achieving,arjevani2020ideal,ye2020multi}.
In this setting, the distributed minimization problem is often reformulated as 
\begin{equation}\label{eq:decentralized}
	\min_{\sqrt{\mW}(x_1,\ldots,x_n)^\top = 0} \left[f(x_1,\ldots,x_n) = \sum_{i=1}^n f_i(x_i)\right],
\end{equation}
where $f_i(x_i)$ is a function stored locally by a computing node $i \in \{1,\ldots,n\}$ and $\mW \in \R^{n\times n}$ is the Laplacian matrix of a graph representing the communication network. The constraint enforces  consensus among the nodes: $x_1=\ldots=x_n$.

One can observe that problem~\eqref{eq:lc} is equivalent to the saddle-point formulation
\begin{equation}\label{eq:lcspp}
	\min_x \max_y f(x) + y^\top\mA x - y^\top b,
\end{equation}
which is another instance of problem~\eqref{eq:main}. State-of-the-art methods  often focus on this formulation instead of directly solving~\eqref{eq:lc}. In particular, \citet{salim2021optimal} and \citet{kovalev2020optimal} obtained optimal algorithms for solving~\eqref{eq:lc} and~\eqref{eq:decentralized} using this saddle-point approach.

\subsection{Bilinear min-max problems}
Unconstrained bilinear saddle-point problems of the form 
\begin{equation}\label{eq:bilinear}
	\min_{x \in \R^{d_x}}\max_{y \in \R^{d_y}} a^\top x + y^\top\mA x - b^\top y
\end{equation}
are another special case of problem~\eqref{eq:main}, one where both $f(x)$ and $g(y)$ are linear functions.
While such  problems do not usually play an important role in practice,  they are often a good testing ground for theoretical purposes \citep{gidel2019negative,azizian2020accelerating, zhang2021don,mokhtari2020unified,daskalakis2018training,liang2019interaction}.

\begin{table}[ht]
	\caption{Comparison of method (APGD, \Cref{apd:alg}) with existing state-of-the-art algorithms for solving problem~\eqref{eq:main} in the 5 different cases described in \cref{sec:related}.}
	\vspace{1em}
	\label{tab:1}
	\centering
	\resizebox*{\linewidth}{!}{
		\begin{NiceTabular}{cc}[code-before = \rowcolor{bgcolor}{2,7,11,15,19}]
			\toprule
			\multicolumn{2}{c}{\bf Strongly-convex-strongly-concave case (\cref{sec:case:scsc})}\\
			\toprule
			\cellcolor{bgcolor} \Cref{apd:alg} &\cellcolor{bgcolor}  $\cO\left(\max\left\{\sqrt{\frac{L_x}{\mu_x}}, \sqrt{\frac{L_y}{\mu_y}},\frac{L_{xy}}{\sqrt{\mu_x\mu_y}}\right\} \log\frac{1}{\epsilon}\right)$
			\\\midrule
			\makecell{Lower bound\\\citet{zhang2021lower}} & $\cO\left(\max\left\{\sqrt{\frac{L_x}{\mu_x}}, \sqrt{\frac{L_y}{\mu_y}},\frac{L_{xy}}{\sqrt{\mu_x\mu_y}}\right\} \log\frac{1}{\epsilon}\right)$
			\\\midrule
			\makecell{DIPPA\\\citet{xie2021dippa}} & $\tilde\cO\left(\max\left\{\sqrt[4]{\frac{L_x^2L_y}{\mu_x^2\mu_y}},\sqrt[4]{\frac{L_xL_y^2}{\mu_x\mu_y^2}},\frac{L_{xy}}{\sqrt{\mu_x\mu_y}}\right\}\log\frac{1}{\epsilon}\right)$
			\\\midrule
			\makecell{Proximal Best Response\\\citet{wang2020improved}} & $\tilde\cO\left(\max\left\{\sqrt{\frac{L_x}{\mu_x}},\sqrt{\frac{L_y}{\mu_y}},\sqrt{\frac{L_{xy}L}{\mu_x\mu_y}}\right\}\log\frac{1}{\epsilon}\right)$
			\\
			\toprule
			\multicolumn{2}{c}{\bf Affinely constrained minimization case (\cref{sec:case:affine})}\\
			\toprule
			\cellcolor{bgcolor} \Cref{apd:alg} & \cellcolor{bgcolor} $\cO\left(
			\frac{L_{xy}}{\mu_{xy}}\sqrt{\frac{L_x}{\mu_x}}\log\frac{1}{\epsilon}\right)$
			\\\midrule
			\makecell{Lower bound\\\citet{salim2021optimal}} & $\cO\left(
			\frac{L_{xy}}{\mu_{xy}}\sqrt{\frac{L_x}{\mu_x}}\log\frac{1}{\epsilon}\right)$
			\\\midrule
			\makecell{OPAPC\\\citet{kovalev2020optimal}} & $\cO\left(
			\frac{L_{xy}}{\mu_{xy}}\sqrt{\frac{L_x}{\mu_x}}\log\frac{1}{\epsilon}\right)$
			\\
			\toprule
			\multicolumn{2}{c}{\bf Strongly-convex-concave case (\cref{sec:case:scc})}\\
			\toprule
			\cellcolor{bgcolor} \Cref{apd:alg} & \cellcolor{bgcolor} $\cO\left(\max\left\{
			\frac{\sqrt{L_xL_y}}{\mu_{xy}},
			\frac{L_{xy}}{\mu_{xy}}\sqrt{\frac{L_x}{\mu_x}},
			\frac{L_{xy}^2}{\mu_{xy}^2}
			\right\}\log\frac{1}{\epsilon}\right)$
			\\\midrule
			Lower bound & N/A
			\\\midrule
			\makecell{Alt-GDA\\\citet{zhang2021don}} & $\cO\left(\max\left\{
			\frac{L^2}{\mu_{xy}^2},
			\frac{L}{\mu_x}
			\right\}\log\frac{1}{\epsilon}\right)$
			\\
			\toprule
			\multicolumn{2}{c}{\bf Bilinear case (\cref{sec:case:bilinear})}\\
			\toprule
			\cellcolor{bgcolor} \Cref{apd:alg} & \cellcolor{bgcolor} $\cO\left(\frac{L_{xy}^2}{\mu_{xy}^2}\log\frac{1}{\epsilon}\right)$
			\\\midrule
			\makecell{Lower bound\\\citet{ibrahim2020linear}} & $\cO\left(\frac{L_{xy}}{\mu_{xy}}\log\frac{1}{\epsilon}\right)$
			\\\midrule
			\citet{azizian2020accelerating} & $\cO\left(\frac{L_{xy}}{\mu_{xy}}\log\frac{1}{\epsilon}\right)$
			\\
			\toprule
			\multicolumn{2}{c}{\bf Convex-concave case (\cref{sec:case:cc})}\\
			\toprule
		\cellcolor{bgcolor} 	\Cref{apd:alg} & \cellcolor{bgcolor}  $\cO\left(\max\left\{
			\frac{\sqrt{L_xL_y}L_{xy}}{\mu_{xy}^2},
			\frac{L_{xy}^2}{\mu_{xy}^2}
			\right\}\log\frac{1}{\epsilon}\right)$
			\\\midrule
			Lower bound & N/A
			\\\bottomrule
	\end{NiceTabular}}
\end{table}

\section{Literature Review and Contributions}


In this work we are interested in algorithms able to solve problem~\eqref{eq:main} with a linear iteration complexity. That is, we are interested in methods that can provably  find an $\epsilon$-accurate solution of problem~\eqref{eq:main} in a number of iterations proportional to $\log\frac{1}{\epsilon}$ (see \Cref{def:epsilon_accurate,def:iteration_complexity}). This is typically achieved when functions $f(x)$ and $g(x)$ are assumed to be strongly convex (see \Cref{def:strongly_convex}). An example of this is the celebrated extragradient method of \citet{korpelevich1976extragradient}.

Recent work has shown that linear iteration complexity can be  achieved also in the less restrictive case when only one of the functions $f(x)$ and $g(x)$ is strongly convex. This was first shown by \citet{du2019linear}, and later improved on by \citet{zhang2021don}. 

\begin{quote}\em However, and this is the starting point of our research, to the best of our knowledge, there are no algorithms with linear iteration complexity in the case when neither  $f(x)$ nor $g(x)$  is strongly convex.
\end{quote}

\subsection{Acceleration}

Loosely speaking, we say that an algorithm is {\em non-accelerated} if its iteration complexity is proportional to at least the first power of the condition numbers associated with the problem,  such as $\nicefrac{L_x}{\mu_x}$ and $\nicefrac{L_y}{\mu_y}$, where $L_x$ and $L_y$ are  smoothness constants, and $\mu_x$ and $\mu_y$ are  strong convexity constants (see \Cref{ass:f} and \Cref{ass:g}). In contrast, the iteration complexity of an {\em accelerated} algorithm is proportional to the square root of such condition numbers, e.g.,  $\sqrt{\nicefrac{L_x}{\mu_x}}$ and $\sqrt{\nicefrac{L_y}{\mu_y}}$.

There were several recent attempts to design accelerated algorithms for solving problem~\eqref{eq:main} \citep{xie2021dippa,wang2020improved,alkousa2020accelerated}. These attempts rely on {\em stacking multiple algorithms on top of each other}, and result in complicated methods. For example, \citet{lin2020near} use a non-accelerated algorithm as a sub-routine for the inexact accelerated proximal-point method. This approach allows them to obtain accelerated algorithms for solving problem~\eqref{eq:main} in a straightforward and tractable way. However, this approach has significant drawbacks: the algorithms obtained this way  have (i)  additional logarithmic factors in their iteration complexity, and (ii) a complex nested structure with the requirement to manually set inner loop sizes, which is a byproduct of the design process based on combining multiple algorithms. This drawback limits the performance of the resulting algorithms in theory, and requires additional fine tuning in practice.

A philosophically different approach to designing such algorithms---one that we adopt in this work---is to attempt to provide a {\em direct} acceleration of a suitable algorithm for solving problem~\eqref{eq:main}, similarly to what \citet{nesterov1983method} did for convex minimization problems. While this technically more demanding, algorithms obtained this way typically don't have the aforementioned drawbacks. Hence, we follow the latter approach in this work.

\subsection{Main contributions}

In this work we propose an Accelerated Primal-Dual Gradient Method (APDG; Algorithm~\ref{apd:alg}) for solving problem ~\eqref{eq:main} and provide a theoretical analysis of its convergence properties (Theorem~\ref{apd:thm}). In particular, we prove the following results.
\begin{itemize}
	\item[(i)] When both functions $f(x)$ and $g(y)$ are strongly convex, Algorithm~\ref{apd:alg} achieves the optimal linear convergence rate, matching the lower bound obtained by \citet{zhang2021lower}. To the best of our knowledge, Algorithm~\ref{apd:alg} is the first optimal algorithm in this regime.
	\item[(ii)] We establish  linear convergence of Algorithm~\ref{apd:alg} in the case when {\em only one} of the functions $f(x)$ or $g(y)$ is strongly convex, and $\mA$ is a full row or full column rank matrix, respectively. This improves upon the results provided by \citet{du2019linear,zhang2021don}.
	\item[(iii)]  We establish  linear convergence of the Algorithm~\ref{apd:alg} in the case when {\em neither} of the functions $f(x)$ nor $g(y)$ is strongly convex, and the matrix $\mA$ is square and full rank. To the best of our knowledge, Algorithm~\ref{apd:alg} is the first algorithm achieving linear convergence in this setting.
\end{itemize}
\Cref{tab:1} provides a brief comparison of the complexity of \Cref{apd:alg} (\Cref{apd:thm}) with the current state of the art. Please refer to \cref{sec:related} for a detailed discussion of this result and comparison with  related work.

\subsection{General min-max problem and additional contributions}
In our work we also consider the saddle-point problem 
\begin{equation}\label{eq:main2}
	\min_{x\in\R^{d_x}} \max_{y \in \R^{d_y}} F(x,y),
\end{equation}
where $F(x,y)\colon \R^{d_x}\times \R^{d_y}\rightarrow\R$ is a smooth function, which is convex in $x$ and concave in $y$. One can observe that the main problem~\eqref{eq:main} is a special case of this more general problem~\eqref{eq:main2}.

As an additional contribution, we propose a Gradient Descent-Ascent Method with Extrapolation  (GDAE; \Cref{gda:alg}) for solving the general convex-concave saddle-point problem~\eqref{eq:main2}, and provide a theoretical analysis of its convergence properties (Theorem~\ref{gda:thm}). 
\begin{enumerate}
	\item[(i)] When the function $F(x,y)$ is  strongly convex in $x$ and strongly concave in $y$, Algorithm~\ref{gda:alg} achieves a linear convergence rate, which recovers the convergence result of \citet{cohen2020relative}.
	\item[(ii)] Under  certain assumptions on the way the variables $x$ and $y$ are coupled by the function $F(x,y)$, we establish  linear convergence of Algorithm~\ref{gda:alg} in the case when the function $F(x,y)$ is strongly-convex-concave, convex-strongly-concave, or even just convex-concave. To the best of our knowledge, Algorithm~\ref{gda:alg} is the first algorithm achieving linear convergence under such assumptions.
\end{enumerate}
Please refer to \cref{sec:related2} for discussion of  \Cref{gda:thm} and related work.

\section{Basic Definitions and Assumptions}\label{sec:def_ass}

We start by formalizing the notions of smoothness and strong convexity of a function.
\begin{definition}\label{def:strongly_convex}
	Function $h(z) \colon \R^d \rightarrow \R$ is $L$-smooth and $\mu$-strongly convex for $L \geq \mu \geq 0$, if for all $z_1,z_2 \in  \R^d$ the following inequality holds:
	\begin{equation}
		\frac{\mu}{2}\sqn{z_1 - z_2} \leq \bg_h (z_1,z_2) \leq \frac{L}{2}\sqn{z_1 - z_2}.
	\end{equation}
	Above, $\bg_h(z_1,z_2) = h(z_1) - h(z_2) - \<\nabla h(z_2),z_1-z_2>$ is the Bregman divergence associated with the function $h(z)$.
\end{definition}

We are now ready to state the main assumptions that we impose on problem~\eqref{eq:main}. We start with  Assumptions~\ref{ass:f} and~\ref{ass:g} that formalize the strong-convexity and smoothness properties of functions $f(x)$ and $g(y)$.

\begin{assumption}\label{ass:f}
	Function $f(x)$ is $L_x$-smooth and $\mu_x$-strongly convex for $L_x\geq\mu_x\geq 0$.
\end{assumption}

\begin{assumption}\label{ass:g}
	Function $g(y)$ is $L_y$-smooth and $\mu_y$-strongly convex for $L_y\geq\mu_y\geq 0$.
\end{assumption}

Note, that $\mu_x$ and $\mu_y$ are allowed to be zero. That is, both $f(x)$ and $g(y)$ are allowed to be non-strongly convex.

The following assumption formalizes the spectral properties of matrix $\mA$.

\begin{assumption}\label{ass:A}
	There exist constants $L_{xy}> \mu_{xy},\mu_{yx} \geq 0$ such that
	\begin{equation*}
		\begin{split}
			\mu_{xy}^2&\leq \begin{cases}
				\lminp(\mA\mA^\top)  & \nabla g(y) \in \range\mA \;\text{ for all } y \in \R^{d_y}\\
				\lmin(\mA\mA^\top) &\text{otherwise}
			\end{cases}\\
			\mu_{yx}^2& \leq \begin{cases}
				\lminp(\mA^\top\mA)  &\nabla f(x) \in \range \mA^\top \text{ for all } x \in \R^{d_x} \\
				\lmin(\mA^\top\mA) &\text{otherwise}
			\end{cases}\\
			L_{xy}^2& \geq  \lmax(\mA^\top\mA) = \lmax(\mA\mA^\top),
		\end{split}
	\end{equation*}
	where $\lmin(\cdot)$, $\lminp(\cdot)$ and $\lmax(\cdot)$ denote the smallest,  smallest positive and  largest eigenvalue of a matrix, respectively, and $\range$ denotes the range space of a matrix.
\end{assumption}

By $\cS \subset \R^{d_x} \times \R^{d_y}$ we  denote the solution set of problem~\eqref{eq:main}. Note that $(x^*,y^*) \in \cS$ if and only if  $(x^*,y^*)$ satisfies the first-order optimality conditions
\begin{equation}\label{eq:opt}
	\begin{cases}
		\nabla_x F(x^*,y^*) = \nabla f(x^*) + \mA^\top y^* = 0,\\
		\nabla_y F(x^*,y^*) = -\nabla g(y^*) + \mA x^* = 0.
	\end{cases}
\end{equation}

Our main goal is to propose an algorithm for finding a solution to problem~\eqref{eq:main}. Numerical iterative algorithms typically find an approximate solution of a given problem. We formalize this through the following definition.

\begin{definition}\label{def:epsilon_accurate}
	Let the solution set $\cS$ be nonempty. We call a pair of vectors $(x,y) \in \R^{d_x} \times \R^{d_y}$ an $\epsilon$-accurate solution of problem~\eqref{eq:main} for a given accuracy $\epsilon>0$ if it satisfies
	\begin{equation}
		\min_{(x^*,y^*) \in \cS}\max\left\{\sqn{x-x^*},\sqn{y - y^*}\right\} \leq \epsilon.
	\end{equation}
\end{definition}

We also want to propose an {\em efficient} algorithm for solving problem~\eqref{eq:main}. That is, we want to propose an algorithm with the the lowest possible {\em iteration complexity}, which we define next.

\begin{definition}\label{def:iteration_complexity}
	The iteration complexity of an algorithm for solving problem~\eqref{eq:main} is the number of iterations  the algorithm requires to find an $\epsilon$-accurate solution of this problem. At each iteration the algorithm is allowed to perform $\cO(1)$ computations of the gradients $\nabla f(x)$ and $\nabla g(y)$ and  matrix-vector multiplications with matrices $\mA$ and $\mA^\top$.
	
\end{definition}


\section{Accelerated Primal-Dual Gradient Method}

\begin{algorithm*}
	\caption{APDG: Accelerated Primal-Dual Gradient Method}
	\label{apd:alg}
	\begin{algorithmic}[1]
		\State {\bf Input:} $x^0 \in \range \mA^\top, y^0\in \range \mA,$ $\eta_x,\eta_y,\alpha_x,\alpha_y, \beta_x,\beta_y>0$, $\tau_x,\tau_y,\sigma_x,\sigma_y \in (0,1]$, $\theta\in(0,1)$
		\State $x_f^0 = x^0$
		\State $y_f^0 = y^{-1} = y^0$
		\For {$k=0,1,2,\ldots$}
		\State $y_m^{k} = y^{k} + \theta(y^{k} - y^{k-1})$\label{apd:line:y:m}
		\State $x_g^k = \tau_x x^k + (1-\tau_x)x_f^k$\label{apd:line:x:1}
		\State $y_g^k = \tau_y y^k + (1-\tau_y)y_f^k$
		\State  $x^{k+1} = x^k + \eta_x\alpha_x(x_g^k - x^{k}) - \eta_x\beta_x\mA^\top(\mA x^k - \nabla g(y_g^k)) - \eta _x\left(\nabla f(x_g^k) + \mA^\top y_m^{k}\right)$\label{apd:line:x:2}
		\State $y^{k+1} = y^k + \eta_y\alpha_y(y_g^k - y^k) - \eta_y\beta_y\mA(\mA^\top y^k + \nabla f(x_g^k)) - 
		\eta_y (\nabla g(y_g^k)- \mA x^{k+1})$
		\State $x_f^{k+1} = x_g^k + \sigma_x(x^{k+1} - x^k)$\label{apd:line:x:3}
		\State $y_f^{k+1} = y_g^k + \sigma_y(y^{k+1} - y^k)$
		\EndFor
	\end{algorithmic}
\end{algorithm*}

In this section we present the Accelerated Primal-Dual Gradient Method (APDG; Algorithm~\ref{apd:alg}) for solving problem~\eqref{eq:main}. First, we prove an outline of the key ideas used in the development of this algorithm.

\subsection{Algorithm development strategy}\label{sec:development}

First, we observe that problem $\eqref{eq:main}$ is equivalent to the problem of finding a zero of a sum of two monotone operators, $G_1,G_2\colon \R^{d_x} \times \R^{d_y} \rightarrow \R^{d_x} \times \R^{d_y}$, defined as
\begin{align}
	G_1\colon (x,y) &\mapsto (\nabla f(x), \nabla g(y)),\\
	G_2\colon (x,y) &\mapsto (\mA^\top y, -\mA x).
\end{align}
Indeed, $G_1(x^*,y^*) + G_2(x^*,y^*) = 0$ is just another way to write the optimality conditions \eqref{eq:opt}.

\paragraph{The Forward Backward algorithm.}  A natural way to tackle this problem is via {\em Forward Backward algorithm} \citep{Convex-Analysis-and-Monotone-Operator-Theory}, the iterates of which have the form
\begin{equation}\label{apd:eq:FB}
	(x^{k+1},y^{k+1}) = J_{G_2}\left((x^k, y^k) - G_1(x^k,y^k)\right),
\end{equation}
where the operator $J_{G_2}$ is the inverse of the operator $I + G_2$, and $I$ is the identity operator. Note that $J_{G_2}$ can be written as $J_{G_2}\colon (x,y) \mapsto (x^+,y^+)$, where $(x^+,y^+)\in\R^{d_x} \times \R^{d_y}$ is a solution of the linear system
\begin{equation}\label{apd:eq:resolvent}
	\begin{cases}
		x^+ = x - \mA^\top y^+\\
		y^+ = y + \mA x^+
	\end{cases}.
\end{equation}

\paragraph{Linear extrapolation step.}
Next, notice that the computation of operator $J_{G_2}$ requires solving the linear system \eqref{apd:eq:resolvent}. This is expensive\footnote{The solution of \eqref{apd:eq:resolvent} can be written in  a closed form and requires to compute an inverse matrix $(\mI + \mA^\top \mA)^{-1}$ or $(\mI + \mA \mA^\top)^{-1}$, where $\mI$ is the identity matrix of an appropriate size.} and has to be done at each iteration of the Forward Backward algorithm.
Let us instead consider the related problem
\begin{equation}\label{apd:eq:resolvent2}
	\begin{cases}
		x^+ = x - \mA^\top y_m\\
		y^+ = y + \mA x^+
	\end{cases},
\end{equation}
where  $y_m \in \R^{d_y}$ is a newly introduced variable. It's easy to observe that \eqref{apd:eq:resolvent2} is equivalent to \eqref{apd:eq:resolvent} when $y_m = y^+$. Next, notice that choosing $y_m = y$ makes \eqref{apd:eq:resolvent2} easy to solve. However, it turns out that the convergence analysis of an algorithm with this approximation may be challenging \citep{zhang2021don}, especially if we want to combine it with other techniques, such as acceleration. 
Our key  idea is to propose a better alternative: the {\em linear extrapolation step}
\begin{equation}
	y_m = y + \theta(y - y^-),
\end{equation}
where $y^- \in \R^{d_y}$ corresponds to $y$ obtained from the previous iteration of the Forward Backward algorithm, and $\theta \in (0,1]$ is an extrapolation parameter. The linear extrapolation  step was introduced by \citet{chambolle2011first} in the analysis of the Primal-Dual Hybrid Gradient algorithm\footnote{However, the Primal-Dual Hybrid Gradient algorithm is not applicable in our case since it requires to compute the proximal operator of $f(x)$ and $g(y)$ at each iteration. Moreover, \citet{chambolle2011first} established  linear convergence of this algorithm in the strongly-convex-strongly-concave setting only.}.

\paragraph{Nesterov acceleration.} Next, we note that operator $G_1$  is equal to the gradient of the (potential) function $(x,y) \mapsto f(x) + g(y)$ function. This function is smooth and convex due to Assumptions~\ref{ass:f} and~\ref{ass:g}. This allows us to incorporate the {\em Nesterov acceleration} mechanism in the Forward Backward algorithm. Nesterov acceleration is known to be a powerful tool which allows to improve convergence properties of gradient methods \citep{nesterov1983method,nesterov2003introductory}.

\subsection{Convergence of the algorithm}

We are now ready to study the convergence properties of  Algorithm~\ref{apd:alg}. We are interested in the case when the following condition holds:
\begin{equation}\label{eq:mu}
	\min\left\{\max\left\{\mu_x, \mu_{yx}\right\},\max\left\{\mu_y, \mu_{xy}\right\} \right\} > 0.
\end{equation}
In this case one can show that the solution set $\cS$ of problem~\eqref{eq:main} is nonempty. Moreover,  {\em strong duality} holds in this case, as captured by the following lemma.
\begin{lemma}\label{lem}
	Let Assumptions~\ref{ass:f},~\ref{ass:g} and~\ref{ass:A} and condition \eqref{eq:mu} hold. Let $p$ be the optimal value of the primal problem
	\begin{equation}\label{eq:primal}
		p = \min_{x\in\R^{d_x}} \left[P(x) = f(x)  + g^*(\mA x)\right],
	\end{equation}
and	let $d$ be the optimal value of the dual problem
	\begin{equation}\label{eq:dual}
		d = \max_{y\in\R^{d_y}} \left[D(y) = -g(y) - f^*(-\mA^\top y)\right].
	\end{equation}
	Then $p=d$ is finite and $(x^*,y^*) \in \cS$
	if and only if $x^*$ is a solution of the primal problem~\eqref{eq:primal} and $y^*$ is a solution of the dual problem~\eqref{eq:dual}.
\end{lemma}


Under the aforementioned conditions, Algorithm~\ref{apd:alg} achieves  linear convergence. That is, the iteration complexity is proportional to $\log \frac{1}{\epsilon}$.

\begin{theorem}\label{apd:thm}
	Let Assumptions~\ref{ass:f},~\ref{ass:g} and~\ref{ass:A} and condition \eqref{eq:mu} hold. Then
	there exist parameters of Algorithm~\ref{apd:alg} such that its iteration complexity for finding an $\epsilon$-accurate solution of problem~\eqref{eq:main} is 
	\begin{equation}
		\cO\left(\min\left\{T_a,T_b, T_c,T_d\right\} \log \frac{C}{\epsilon}\right),
	\end{equation}
	where  $T_a,T_b, T_c,T_d$ are defined as
	\begin{align*}
		T_a &= \max\left\{\sqrt{\frac{L_x}{\mu_x}}, \sqrt{\frac{L_y}{\mu_y}},\frac{L_{xy}}{\sqrt{\mu_x\mu_y}}\right\},\\
		T_b &=\max\left\{
		\frac{\sqrt{L_xL_y}}{\mu_{xy}},
		\frac{L_{xy}}{\mu_{xy}}\sqrt{\frac{L_x}{\mu_x}},
		\frac{L_{xy}^2}{\mu_{xy}^2}
		\right\},\\
		T_c &= \max\left\{
		\frac{\sqrt{L_xL_y}}{\mu_{yx}},
		\frac{L_{xy}}{\mu_{yx}}\sqrt{\frac{L_y}{\mu_y}},
		\frac{L_{xy}^2}{\mu_{yx}^2}
		\right\},\\
		T_d &= \max\left\{
		\frac{\sqrt{L_xL_y}L_{xy}}{\mu_{xy}\mu_{yx}},
		\frac{L_{xy}^2}{\mu_{yx}^2},
		\frac{L_{xy}^2}{\mu_{xy}^2}
		\right\},
	\end{align*}
	and $C >0$ is some constant, which does not depend on $\epsilon$, but possibly depends on $L_x,\mu_x, L_y,\mu_y, L_{xy}, \mu_{xy}, \mu_{yx}$.
\end{theorem}

\section{Discussion of Theorem~\ref{apd:alg} and Related Work}\label{sec:related}

In this section we comment on the iteration complexity result for Algorithm~\ref{apd:alg} provided in Theorem~\ref{apd:thm}. We consider important and illustrative special cases of this complexity result and draw connections with the existing results in  the literature.

\subsection{Strongly convex and strongly concave case}\label{sec:case:scsc}
In this case $\mu_x,\mu_y > 0$. We can always assume $\mu_{xy} = \mu_{yx} = 0$ in Assumption~\ref{ass:A}. Then, Algorithm~\ref{apd:alg} has  iteration complexity given by
\begin{equation}\label{eq:comp:scsc}
\color{red}	\cO\left(\max\left\{\sqrt{\frac{L_x}{\mu_x}},\sqrt{\frac{L_y}{\mu_y}},\frac{L_{xy}}{\sqrt{\mu_x\mu_y}}\right\} \log\frac{1}{\epsilon}\right).
\end{equation} 
This improves  the current state-of-the-art results
\begin{equation}
	\tilde\cO\left(\max\left\{\sqrt[4]{\frac{L_x^2L_y}{\mu_x^2\mu_y}},\sqrt[4]{\frac{L_xL_y^2}{\mu_x\mu_y^2}},\frac{L_{xy}}{\sqrt{\mu_x\mu_y}}\right\}\log\frac{1}{\epsilon}\right)
\end{equation}
due to  \citet{xie2021dippa}, and 
\begin{equation}
	\tilde\cO\left(\max\left\{\sqrt{\frac{L_x}{\mu_x}},\sqrt{\frac{L_y}{\mu_y}},\sqrt{\frac{L_{xy}L}{\mu_x\mu_y}}\right\} \log\frac{1}{\epsilon}\right),
\end{equation}
due to \citet{wang2020improved}, where $\tilde\cO(\cdot)$ hides additional logarithmic factors, and $L = \max\{L_x,L_y,L_{xy}\}$.
Moreover, our result \eqref{eq:comp:scsc} matches the lower complexity bound provided by \citet{zhang2021lower}. Hence, {\em Algorithm~\ref{apd:alg} is optimal in this regime.} To the best of our knowledge, Algorithm~\ref{apd:alg} is the first algorithm which achieves the lower complexity bound \eqref{eq:comp:scsc} for smooth and strongly-convex-strongly-concave saddle-point problems with bilinear coupling.

\subsection{Affinely-constrainted minimization case}\label{sec:case:affine}
In this case $\mu_x > 0 $ and $\mu_y = 0$. Firstly, we consider the case when $L_y = 0$, i.e., $g(y)$ is a linear function. Then, problem~\eqref{eq:main} is equivalent to the smooth and strongly-convex  affinely-constrained minimization problem~\eqref{eq:lc}.
Algorithm~\ref{apd:alg} enjoys the linear convergence rate
\begin{equation}\label{eq:comp:lc}
\color{red}	\cO\left(\max\left\{
	\frac{L_{xy}}{\mu_{xy}}\sqrt{\frac{L_x}{\mu_x}},
	\frac{L_{xy}^2}{\mu_{xy}^2}
	\right\}\log\frac{1}{\epsilon}\right),
\end{equation}
where $\mu_{xy} = \lminp(\mA\mA^\top) > 0$ due to Assumption~\ref{ass:A}. This result recovers the complexity of the APAPC algorithm \citep{kovalev2020optimal}. It is possible to incorporate the Chebyshev acceleration mechanism \citep{arioli2014chebyshev} into Algorithm~\ref{apd:alg} for solving problem~\eqref{eq:lc} to obtain the improved complexity
\begin{equation}\label{eq:comp:opt_lc}
\color{red}	\cO\left(
	\frac{L_{xy}}{\mu_{xy}}\sqrt{\frac{L_x}{\mu_x}}\log\frac{1}{\epsilon}\right).
\end{equation}
This matches the complexity of the OPAPC algorithm of \citet{kovalev2020optimal,salim2021optimal}, which was shown to be optimal \citep{salim2021optimal,scaman2017optimal}.

\subsection{Strongly convex and concave case}\label{sec:case:scc}
We also allow $L_y>0$, i.e., function $g(y)$ is a general, not necessarily linear, smooth and convex function. It is often possible that $\mu_{xy} > 0$ due to Assumption~\ref{ass:A}; for instance, when $\mA$ is a full row rank matrix. Then, Algorithm~\ref{apd:alg} enjoys the following linear iteration complexity:
\begin{equation}\label{eq:comp:scc}
\color{red}	\cO\left(\max\left\{
	\frac{\sqrt{L_xL_y}}{\mu_{xy}},
	\frac{L_{xy}}{\mu_{xy}}\sqrt{\frac{L_x}{\mu_x}},
	\frac{L_{xy}^2}{\mu_{xy}^2}
	\right\}\log\frac{1}{\epsilon}\right).
\end{equation}
This case was previously studied by \citet{du2019linear, du2017stochastic, zhang2021don}. \citet{du2019linear} provided an analysis for an algorithm called Sim-GDA, and established its iteration complexity
\begin{equation}
	\cO\left(\max\left\{
	\frac{L_x^3L_yL_{xy}^2}{\mu_x^2\mu_{xy}^4},
	\frac{L_x^3L_{xy}^4}{\mu_x^3\mu_{xy}^4}
	\right\}\log\frac{1}{\epsilon}\right).
\end{equation}
This result is substantially worse than our complexity \eqref{eq:comp:scc}; possibly due to a suboptimal analysis. Subsequently, \citet{zhang2021don} provided an improved analysis for the Sim-GDA algorithm, obtaining the complexity 
\begin{equation}
	\cO\left(\max\left\{
	\frac{L^3}{\mu_x\mu_{xy}^2},
	\frac{L^2}{\mu_x^2}
	\right\}\log\frac{1}{\epsilon}\right).
\end{equation}
They also studied the Alt-GDA method, obtaining the complexity
\begin{equation}
	\cO\left(\max\left\{
	\frac{L^2}{\mu_{xy}^2},
	\frac{L}{\mu_x}
	\right\}\log\frac{1}{\epsilon}\right),
\end{equation}
where $L = \max\{L_x,L_y,L_{xy}\}$. However, these results are local, i.e., they are valid only if the initial iterates of these algorithms are close enough to the solution of problem \eqref{eq:main}. Moreover, these results are still worse than our rate \eqref{eq:comp:scc} because Sim-GDA and Alt-GDA do not utilize the Nesterov acceleration mechanism, while our Algorithm~\ref{apd:alg} does.

\subsection{Bilinear case}\label{sec:case:bilinear}
In this case $\mu_x = \mu_y = L_x = L_y = 0$. That is, functions $f(x)$ and $g(y) $ are linear. Then, problem~\eqref{eq:main} turns into the bilinear  min-max problem \eqref{eq:bilinear}, and $\mu_{xy}^2 = \mu_{yx}^2 = \lminp(\mA^\top\mA) > 0$ due to Assumption~\ref{ass:A}. The iteration complexity of Algorithm~\ref{apd:alg} becomes
\begin{equation}\label{eq:comp:bilinear}
\color{blue}	\cO\left(\frac{L_{xy}^2}{\mu_{xy}^2}\log\frac{1}{\epsilon}\right).
\end{equation}
This recovers the results of \citet{daskalakis2018training,liang2019interaction,gidel2018variational,gidel2019negative,mishchenko2020revisiting,mokhtari2020unified} for the bilinear min-max problem \eqref{eq:bilinear}. However, this result is worse than the complexity lower bound 
\begin{equation}
	\cO\left(\frac{L_{xy}}{\mu_{xy}}\log\frac{1}{\epsilon}\right),
\end{equation}
obtained in the work of \citet{ibrahim2020linear}, which was reached  by \citet{azizian2020accelerating}\footnote{We provide these results for completeness. The result of \citet{azizian2020accelerating} is better than our result~\eqref{eq:comp:bilinear} for Algorithm~\ref{apd:alg} because they specifically focus on solving the bilinear min-max problem~\eqref{eq:bilinear}, while Algorithm~\ref{apd:alg} aims to solve the much more general convex-concave saddle-point problem~\eqref{eq:main}.}.

\begin{algorithm*}[!htb]
	\caption{GDAE: Gradient Descent-Ascent with Extrapolation}
	\label{gda:alg}
	\begin{algorithmic}
		\State {\bf Input:} $x^0 \in \R^{d_x}$, $y^0 \in \R^{d_y}$, $\eta_x, \eta_y > 0$, $\theta \in (0,1)$
		\State $x^{-1} = x^0$
		\State $y^{-1} = y^0$
		\For{$k=0,1,2,\ldots$}
		\State $x^{k+1} = x^k - \eta_x \nabla_x F(x^k,y^k) - \eta_x\theta(\nabla_x F(x^{k-1},y^k) - \nabla_x F(x^{k-1},y^{k-1}) )$\label{gda:line:x}
		\State $y^{k+1} = y^k + \eta_y \nabla_y F(x^{k+1}, y^k)$\label{gda:line:y}
		\EndFor
	\end{algorithmic}
\end{algorithm*}

\subsection{Convex-concave case}\label{sec:case:cc}
In this case $\mu_y = \mu_x = 0$. It is often possible that $\mu_{xy} = \mu_{yx} > 0$ due to Assumption~\ref{ass:A}, for example, when $\mA$ is a square and full rank matrix. Then, the iteration complexity of Algorithm~\ref{apd:alg} becomes
\begin{equation}\label{eq:comp:cc}
\color{red}	\cO\left(\max\left\{
	\frac{\sqrt{L_xL_y}L_{xy}}{\mu_{xy}^2},
	\frac{L_{xy}^2}{\mu_{xy}^2}
	\right\}\log\frac{1}{\epsilon}\right),
\end{equation}
which is still linear.
This complexity result generalizes the result \eqref{eq:comp:bilinear} for bilinear min-max problems   as it allows for general, not necessarily linear, convex and smooth functions $f(x)$ and $g(x)$.
To the best of our knowledge, Algorithm~\ref{apd:alg} is the first algorithm which can achieve linear convergence for smooth and non-strongly convex non-strongly concave min-max problems with bilinear coupling.

\section{A Novel Gradient Method for General Convex-Concave Saddle-Point Problems}

In this section we present a new method---Gradient Descent-Ascent Method with Extrapolation (GDAE; Algorithm~\ref{gda:alg})---for solving problem~\eqref{eq:main2}. 

\subsection{Assumptions and definitions}

First, we state the main assumptions that we impose on problem~\eqref{eq:main2}. 
\begin{assumption}\label{ass:F1}
	Function $F(x,y)$ is $L_x$-smooth and $\mu_x$-strongly convex in $x$ and $L_y$-smooth and $\mu_y$-strongly concave in $y$, where $L_x\geq \mu_x\geq 0$, $L_y\geq \mu_y\geq 0$.
\end{assumption}

Assumption~\ref{ass:F1} generalizes the smoothness and strong convexity Assumptions~\ref{ass:f} and~\ref{ass:g} imposed on problem~\eqref{eq:main}.

\begin{assumption}\label{ass:F2}
	There exists a constant $L_{xy} > 0$ such that for all $x,x_1,x_2 \in \R^{d_x}$ and $y,y_1,y_2 \in \R^{d_y}$, the following inequalities hold:
	\begin{equation}\label{eq:F2}
		\begin{split}
			\norm{\nabla_x F(x,y_1) - \nabla_x F(x,y_2)} &\leq L_{xy}\norm{y_1 - y_2},\\
			\norm{\nabla_y F(x_1,y) - \nabla_yF(x_2,y)} &\leq L_{xy}\norm{x_1 - x_2}.
		\end{split}
	\end{equation}
\end{assumption}

\begin{assumption}\label{ass:F3}
	There exist constants $\mu_{xy}, \mu_{yx} \geq 0$ such that for all $x,x_1,x_2 \in \R^{d_x}$ and $y,y_1,y_2 \in \R^{d_y}$, the following inequalities hold:
	\begin{equation}\label{eq:F3}
		\begin{split}
			\norm{\nabla_x F(x,y_1) - \nabla_x F(x,y_2)} &\geq \mu_{xy}\norm{y_1 - y_2},\\
			\norm{\nabla_y F(x_1,y) - \nabla_yF(x_2,y)} &\geq \mu_{yx}\norm{x_1 - x_2}.
		\end{split}
	\end{equation}
\end{assumption}

Assumptions~\ref{ass:F2} and~\ref{ass:F3} combined form a generalized  version of Assumption~\ref{ass:A} for problem~\eqref{eq:main2}. Indeed, if one assumes that \eqref{eq:F2} and \eqref{eq:F3} hold for problem~\eqref{eq:main}, then the following inequalities hold
\begin{equation}
	\begin{split}
		\mu_{xy}^2 &\leq \lmin(\mA\mA^\top) \leq L_{xy}^2,\\
		\mu_{yx}^2 &\leq \lmin(\mA^\top\mA) \leq L_{xy}^2,
	\end{split}
\end{equation}
which can be seen as a simplified version of Assumption~\ref{ass:A}.

Next, we recall several basic definitions.
Similarly to \cref{sec:def_ass}, by $\cS \subset \R^{d_x}\times \R^{d_y}$ we denote the solution set of problem~\eqref{eq:main2}. Note that $(x^*,y^*) \in \cS$ if and only if $(x^*,y^*)$ satisfies the optimality conditions
\begin{equation}\label{eq:opt2}
	\begin{cases}
		\nabla_x F(x^*,y^*)  = 0,\\
		\nabla_y F(x^*,y^*)  = 0.
	\end{cases}
\end{equation}
We also use notions of iteration complexity for achieving an $\epsilon$-accurate solution analogous to \Cref{def:epsilon_accurate,def:iteration_complexity}, respectively.

\subsection{Algorithm development}
We now present the main ingredients and intuition behind the development of our method.


\paragraph{Implicit gradient descent-ascent.}
First, we recall the iterations of the Forward-Backward algorithm \eqref{apd:eq:FB}, which can be written in the form
\begin{equation} \label{eq:buh9fd8hf8d}
	\begin{cases}
		x^{+} = x - \eta_x \nabla f (x) - \eta_x \mA^\top y^{+}\\
		y^{+} = y - \eta_y \nabla g(y) + \eta_y \mA x^{+}
	\end{cases},
\end{equation}
where $\eta_x,\eta_y > 0$ are stepsizes. Iterations \eqref{eq:buh9fd8hf8d} can also be written in terms of the gradients $\nabla_x F(x,y)$ and $\nabla_y F(x,y)$,
\begin{equation}\label{gda:eq:implicit}
	\begin{cases}
		x^{+} = x - \eta_x \nabla_x F(x, y^+)\\
		y^{+} = y + \eta_y \nabla_y F(x^+,y)
	\end{cases},
\end{equation}
which makes the method applicable to the general problem~\eqref{eq:main2}.

Iterations \eqref{gda:eq:implicit} were the foundation for the development of Algorithm~\ref{apd:alg} for solving problem~\eqref{eq:main}, with strong convergence properties established by Theorem~\ref{apd:thm}. Hence, we expect that this approach would work for solving the more general problem~\eqref{eq:main2}. However, \eqref{gda:eq:implicit} is an implicit algorithm and can't be applied in its current state.

\paragraph{Gradient extrapolation.}

In analogy to the development of Algorithm~\ref{apd:alg}, we want to find a good approximation of the implicit iterations \eqref{gda:eq:implicit}. A naive solution would be using the approximation
\begin{equation}
	\begin{cases}
		x^{+} = x - \eta_x \nabla_x F(x, y_m)\\
		y^{+} = y + \eta_y \nabla_y F(x^+,y)
	\end{cases},
\end{equation}
where $y_m \approx y^+$. Similarly to Section~\ref{sec:development}, we could use $y_m = y$, which would lead to the Alt-GDA algorithm \citep{zhang2021don}, or $y_m = y + \theta(y - y^-)$, which is a linear extrapolation step \citep{chambolle2011first}. 

The linear extrapolation step with $\theta = 1$ is based on the ``assumption'' that $y^+\approx y_m = y + (y - y^-)$, or equivalently, $y^+ - y\approx y - y^-$. We can use a similar intuition for the gradients $\nabla_x F(x,y)$ rather than the iterates $y$. In particular, we ``assume'' that
\begin{equation*}
	\nabla_x F(x,y^+) - \nabla_x F(x,y) \approx \nabla_x F(x^-,y) - \nabla_x F(x^-,y^-),
\end{equation*}
or equivalently,
\begin{equation*}
	\begin{cases}
		\nabla_x F(x,y^+)\approx  \Delta_x\\
		\Delta_x = \nabla_x F(x,y)+  (\nabla_x F(x^-,y) - \nabla_x F(x^-,y^-))
	\end{cases}.
\end{equation*}
This intuition leads to the following novel update rule, which we call {\em gradient extrapolation step}:
\begin{equation*}
	\begin{cases}
		\Delta_x =  \nabla_x F(x,y) + \theta(\nabla_x F(x^-,y) - \nabla_x F(x^-,y^-))\\
		x^+ = x - \eta_x \Delta_x
	\end{cases}.
\end{equation*}
Above, $\theta \in (0,1]$ is the extrapolation parameter.
We use this gradient extrapolation step together with the update rule for $y$ from \eqref{gda:eq:implicit} in the design of our Algorithm~\ref{gda:alg}.

\subsection{Convergence of \Cref{gda:alg} and related work}\label{sec:related2}
We now present Theorem~\ref{gda:thm}, which establishes  linear convergence rate for Algorithm~\ref{gda:alg} under Assumptions~\ref{ass:F1},~\ref{ass:F2} and~\ref{ass:F3}.
\begin{theorem}\label{gda:thm}
	Let Assumptions~\ref{ass:F1},~\ref{ass:F2} and~\ref{ass:F3} and condition \eqref{eq:mu} hold. Then
	there exist parameters of Algorithm~\ref{gda:alg} such that the iteration complexity for finding an $\epsilon$-accurate solution of problem~\eqref{eq:main2} is 
	\begin{equation}
		\cO\left(\min\left\{T_a,T_b, T_c,T_d\right\} \log \frac{C}{\epsilon}\right),
	\end{equation}
	where  $T_a,T_b, T_c,T_d$ are defined as
	\begin{align*}
		T_a &= \max\left\{\frac{L_x}{\mu_x}, \frac{L_y}{\mu_y},\frac{L_{xy}}{\sqrt{\mu_x\mu_y}}\right\},\\
		T_b &=\max\left\{
		\frac{L_x}{\mu_x},
		\frac{L_xL_y}{\mu_{xy}^2},
		\frac{L_{xy}^2}{\mu_{xy}^2}
		\right\},\\
		T_c &= \max\left\{
		\frac{L_y}{\mu_y},
		\frac{L_xL_y}{\mu_{yx}^2},
		\frac{L_{xy}^2}{\mu_{yx}^2}
		\right\},\\
		T_d &= \max\left\{
		\frac{L_xL_y}{\mu_{xy}^2},
		\frac{L_xL_y}{\mu_{yx}^2},
		\frac{L_{xy}^2}{\mu_{xy}^2},
		\frac{L_{xy}^2}{\mu_{yx}^2}\right\},
	\end{align*}
	and $C >0$ is some constant, which does not depend on $\epsilon$, but possibly depends on $L_x,\mu_x, L_y,\mu_y, L_{xy}, \mu_{xy}, \mu_{yx}$.
\end{theorem}

Consider the case when $\mu_x, \mu_y > 0$. In this case the iteration complexity of \Cref{gda:alg} becomes
\begin{equation}
\color{red}	\cO\left(\max\left\{\frac{L_x}{\mu_x}, \frac{L_y}{\mu_y},\frac{L_{xy}}{\sqrt{\mu_x\mu_y}}\right\}\log\frac{1}{\epsilon}\right).
\end{equation}
This recovers the result of \citet{cohen2020relative}. Moreover, when $\mu_x = \mu_y$, this result recovers the complexity of solving problem~\eqref{eq:main2} by a number of known algorithms, including the extragradient method \citep{korpelevich1976extragradient}, optimistic gradient method \citep{daskalakis2018training,gidel2018variational}, and the dual extrapolation method \citep{nesterov2006solving}.

Finally, consider then opposite case when at least one of the constants $\mu_x$ and $\mu_y$ is zero. To the best of our knowledge, there are no algorithms that can achieve a linear convergence. However, \Cref{gda:alg} can still achieve linear iteration complexity provided that  condition~\eqref{eq:mu} is satisfied.

\bibliography{reference}
\bibliographystyle{icml2022}

\newpage
\appendix
\onecolumn

\part*{Appendix}

In  \Cref{sec:AppC} we provide a proof of \Cref{lem}, in \Cref{sec:AppA} we provide a proof of \Cref{apd:thm}, and in \Cref{sec:AppB} we provide a proof of \Cref{gda:thm}.

\section{Proof of \Cref{lem}}\label{sec:AppC}

{\bf Part 1.} Let us first show that primal problem~\eqref{eq:primal} has at least a single solution $x^* \in \R^{d_x}$.

	Condition $\eqref{eq:mu}$ implies that $\max\{\mu_x,\mu_{yx}\} > 0$. If $\mu_x > 0$ then function $P(x)$ is obviously strongly convex and primal problem indeed has a solution. Consider the opposite case $\mu_x = 0$. Then $\mu_{yx} > 0$ due to condition \eqref{eq:mu}.
	
	\Cref{ass:A} and $\mu_{yx} >0$ imply that $\nabla f(x) \in \range \mA^\top$ for all $x \in \R^{d_x}$. Hence, 
	\begin{equation}
		f(x + x') = f(x) \text{ for all } x\in\R^{d_x}, x' \in \ker \mA.
	\end{equation}
	Using the definition of $P(x)$ we get
	\begin{align*}
		P(x + x') &= f(x + x') + g^*(\mA(x + x'))
		\\&=
		f(x) + g^*(\mA x) \\&= P(x) 
	\end{align*}
	for all  $x\in\R^{d_x}, x' \in \ker \mA$. From this one can conclude that
	\begin{equation*}
		\min_{x \in \R^{d_x}} P(x)  = \min_{x \in x^0 + \range \mA^\top} P(x).
	\end{equation*}
	for any vector $x^0 \in \R^{d_x}$. From the definition of $P(x)$ it follows that $P(x)$ is $\mu_{yx}$-strongly convex on any affine space $x^0 + \range \mA^\top$ for arbitrary $x^0 \in \R^{d_x}$.
	Hence, problem $\min_{x \in x^0 + \range \mA^\top} P(x)$ has a unique solution and primal problem $\min_{x \in \R^{d_x}} P(x)$ has at least a single solution $x^*$.

{\bf Part 2.} Let us show that there exists $y^* \in \R^{d_y}$ such that $(x^*,y^*) \in \cS$, i.e., $(x^*,y^*)$ satisfy optimality conditions  \eqref{eq:opt}.

	Let us show that $-\nabla f(x^*) \in \mA^\top \partial g^*(\mA x^*)$. We use condition \eqref{eq:mu} which implies $\max\{\mu_y,\mu_{xy}\} > 0$. If $\mu_y > 0$, then function $g^*(y)$ is smooth and our statement is trivial. Consider the opposite case $\mu_y = 0$.  Then $\mu_{xy} > 0$ due to condition~\eqref{eq:mu}.
	
	\Cref{ass:A} and $\mu_{xy} >0$ imply that $\nabla g(y) \in \range \mA$ for all $y \in \R^{d_y}$. Hence, $\dom g^*(\cdot) \subset \range \mA$. Let $h(x) = g^*(\mA x)$. From standard theory it follows that $-\nabla f(x^*) \in \partial h(x^*)$ or
	\begin{equation*}
		h(x) \geq h(x^*) - \<\nabla f(x^*), x - x^* > \text{ for all } x \in \R^{d_x},
	\end{equation*}
	From this one can conclude that
	\begin{equation*}
		\<\nabla f(x^*),x - x^*> \geq 0 \text{ for all } x \in x^* + \ker \mA,
	\end{equation*}
	which implies $\nabla f(x^*) \in (\ker \mA)^\perp = \range \mA^\top$. Hence, there exists vector $y^* \in \R^{d_y}$ such that $-\nabla f(x^*) = \mA^\top y^*$. Now, we can write
	\begin{equation*}
		h(x) \geq h(x^*) + \<\mA^\top y^*, x - x^* > \text{ for all } x \in \R^{d_x},
	\end{equation*}
	which is equivalent to
	\begin{equation*}
		g^*(\mA x) \geq g^*(\mA x^*) + \<y^*, \mA x - \mA x^* > \text{ for all } x \in \R^{d_x}.
	\end{equation*}
	The latter can be written as
	\begin{equation*}
		g^*(y) \geq g^*(\mA x^*) + \<y^*, y - \mA x^* > \text{ for all } y \in \range \mA.
	\end{equation*}
	But $\dom g^*(\cdot) \subset \range \mA$, which means that $g^*(y) = +\infty$ for all $y \notin \range \mA$.
	This implies
	\begin{equation*}
		g^*(y) \geq g^*(\mA x^*) + \<y^*, y - \mA x^* > \text{ for all } y \in \R^{d_y},
	\end{equation*}
	which is a definition of $y^* \in \partial g^*(\mA x^*)$. An equivalent for this is $\nabla g(y^*) = \mA x^*$, which together with $-\nabla f(x^*) = \mA^\top y^*$ form optimality condition~\eqref{eq:opt}.
	
	{\bf Part 3.}
	We showed that there exists a pair of vectors $(x^*,y^*) \in \R^{d_x} \times \R^{d_y}$ which is a saddle point of the function $F(x,y)$ in problem~\eqref{eq:main}. Hence, strong duality holds and proof of the rest of \Cref{lem} is trivial.
\qed

\newpage

\section{Proof of Theorem~\ref{apd:thm}} \label{sec:AppA}

\begin{lemma}
	There exists a solution $(x^*,y^*) \in \cS$ of the problem~\eqref{eq:main} such that for all $k=0,1,2,\ldots$ the iterates of \Cref{apd:alg} satisfy
	\begin{equation}\label{apd:eq:mu:2}
		\begin{split}
			\norm{\mA(x^k - x^*)} &\geq \mu_{yx}\norm{x^k - x^*},\\
			\norm{\mA^\top(y^k - y^*)} &\geq \mu_{xy}\norm{y^k - y^*}.
		\end{split}
	\end{equation}
\end{lemma}
\begin{proof}
	The proof of this lemma is a trivial extension of the derivations from the proof of \Cref{lem}.
\end{proof}

\begin{lemma}
	Let $\tau_x$ be defined as
	\begin{equation}
		\tau_x = (\sigma_x^{-1} + 1/2)^{-1}.
	\end{equation}
	Let $\alpha_x$ be defined as
	\begin{equation}
		\alpha_x = \mu_x.
	\end{equation}
	Let $\beta_x$ be defined as
	\begin{equation}
		\beta_x = \min\left\{\frac{1}{2L_y}, \frac{1}{2\eta_x L_{xy}^2}\right\}.
	\end{equation}
	Then, the following inequality holds:
	\begin{equation}\label{apd:eq:x}
		\begin{split}
			\frac{1}{\eta_x}\sqn{x^{k+1} - x^*}
			&\leq
			\left(\frac{1}{\eta_x} - \mu_x - \beta_x\mu_{yx}^2\right)\sqn{x^{k} - x^*}
			+
			\left(\mu_x + L_x\sigma_x -\frac{1}{2\eta_x}\right)
			\sqn{x^{k+1} - x^k}
			\\&
			+
			\bg_g(y_g^k,y^*)
			-
			\bg_f(x_g^k,x^*) 
			-
			\frac{2}{\sigma_x}\bg_f(x_f^{k+1},x^*) 
			+
			\left(\frac{2}{\sigma_x} - 1\right)\bg_f(x_f^k,x^*) 
			\\&-
			2\<\mA^\top (y_m^k - y^*),x^{k+1} - x^*>.
		\end{split}
	\end{equation}
\end{lemma}
\begin{proof}
	Using Line~\ref{apd:line:x:2} of the Algorithm~\ref{apd:alg} we get
	\begin{align*}
		\frac{1}{\eta_x}\sqn{x^{k+1} - x^*}
		&=
		\frac{1}{\eta_x}\sqn{x^{k} - x^*} + \frac{2}{\eta_x}\<x^{k+1} - x^k,x^{k+1} - x^*> - \frac{1}{\eta_x}\sqn{x^{k+1} - x^k}
		\\&=
		\frac{1}{\eta_x}\sqn{x^{k} - x^*} + 2\alpha_x\<x_g^k - x^k,x^{k+1}- x^*>
		-
		2\beta_x\<\mA^\top(\mA x^k - \nabla g(y_g^k)),x^{k+1} - x^*>
		\\&-
		2\<\nabla f(x_g^k) + \mA^\top y_m^k,x^{k+1} - x^*>
		-
		\frac{1}{\eta_x}\sqn{x^{k+1} - x^k}.
	\end{align*}
	Using the parallelogram rule we get
	\begin{align*}
		\frac{1}{\eta_x}\sqn{x^{k+1} - x^*}
		&=
		\frac{1}{\eta_x}\sqn{x^{k} - x^*}
		+
		\alpha_x\left(\sqn{x_g^k - x^*}  - \sqn{x_g^k - x^{k+1}} - \sqn{x^{k} - x^*}+\sqn{x^{k+1} - x^k}\right)
		\\&-
		2\beta_x\<\mA x^k - \nabla g(y_g^k),\mA(x^{k+1} - x^*)>
		-
		2\<\nabla f(x_g^k) + \mA^\top y_m^k,x^{k+1} - x^*>
		-
		\frac{1}{\eta_x}\sqn{x^{k+1} - x^k}.
	\end{align*}
	Using the optimality condition $\nabla g(y^*) = \mA x^*$, which follows from \eqref{eq:opt}, and the parallelogram rule we get
	\begin{align*}
		\frac{1}{\eta_x}\sqn{x^{k+1} - x^*}
		&=
		\frac{1}{\eta_x}\sqn{x^{k} - x^*}
		+
		\alpha_x\left(\sqn{x_g^k - x^*}  - \sqn{x_g^k - x^{k+1}} - \sqn{x^{k} - x^*}+\sqn{x^{k+1} - x^k}\right)
		\\&+
		\beta_x\left(\sqn{\mA(x^{k+1} - x^k)}  - \sqn{\mA(x^k - x^*)} + \sqn{\nabla g(y_g^k) - \nabla g(y^*)} - \sqn{\nabla g (y_g^k) - \mA(x^{k+1})}\right)
		\\&-
		2\<\nabla f(x_g^k) + \mA^\top y_m^k,x^{k+1} - x^*>
		-
		\frac{1}{\eta_x}\sqn{x^{k+1} - x^k}.
	\end{align*}
	Using Assumption~\ref{ass:A}, equation~\ref{apd:eq:mu:2} and $L_y$-smoothness of $g$ we get
	\begin{align*}
		\frac{1}{\eta_x}\sqn{x^{k+1} - x^*}
		&\leq
		\frac{1}{\eta_x}\sqn{x^{k} - x^*}
		+
		\alpha_x\sqn{x_g^k - x^*}  - \alpha_x\sqn{x^{k} - x^*}+\alpha_x\sqn{x^{k+1} - x^k}
		\\&+
		\beta_xL_{xy}^2\sqn{x^{k+1} - x^k}  - \beta_x\mu_{yx}^2\sqn{x^k - x^*} + 2\beta_xL_y\bg_g(y_g^k,y^*)
		\\&-
		2\<\nabla f(x_g^k) + \mA^\top y_m^k,x^{k+1} - x^*>
		-
		\frac{1}{\eta_x}\sqn{x^{k+1} - x^k}
		\\&=
		\left(\frac{1}{\eta_x} - \alpha_x - \beta_x\mu_{yx}^2\right)\sqn{x^{k} - x^*}
		+
		\left(\beta_xL_{xy}^2 + \alpha_x-\frac{1}{\eta_x}\right)
		\sqn{x^{k+1} - x^k}
		\\&+
		2\beta_xL_y\bg_g(y_g^k,y^*)
		+
		\alpha_x\sqn{x_g^k - x^*}
		-
		2\<\nabla f(x_g^k) + \mA^\top y_m^k,x^{k+1} - x^*>.
	\end{align*}
	Using the optimality condition $\nabla f(x^*) + \mA^\top y^* = 0$, which follows from \eqref{eq:opt}, we get
	\begin{align*}
		\frac{1}{\eta_x}\sqn{x^{k+1} - x^*}
		&\leq
		\left(\frac{1}{\eta_x} - \alpha_x - \beta_x\mu_{yx}^2\right)\sqn{x^{k} - x^*}
		+
		\left(\beta_xL_{xy}^2 + \alpha_x-\frac{1}{\eta_x}\right)
		\sqn{x^{k+1} - x^k}
		+
		2\beta_xL_y\bg_g(y_g^k,y^*)
		\\&+
		\alpha_x\sqn{x_g^k - x^*}
		-
		2\<\nabla f(x_g^k) - \nabla f(x^*),x^{k+1} - x^*>
		-
		2\<\mA^\top (y_m^k - y^*),x^{k+1} - x^*>
		\\&=
		\left(\frac{1}{\eta_x} - \alpha_x - \beta_x\mu_{yx}^2\right)\sqn{x^{k} - x^*}
		+
		\left(\beta_xL_{xy}^2 + \alpha_x-\frac{1}{\eta_x}\right)
		\sqn{x^{k+1} - x^k}
		\\&+
		2\beta_xL_y\bg_g(y_g^k,y^*)
		+
		\alpha_x\sqn{x_g^k - x^*}
		-
		2\<\nabla f(x_g^k) - \nabla f(x^*),x^{k+1}- x^k + x^k - x_g^k + x_g^k - x^*>
		\\&-
		2\<\mA^\top (y_m^k - y^*),x^{k+1} - x^*>
	\end{align*}
	Using $\mu_y$-strong convexity of $f$ and Lines~\ref{apd:line:x:1} and~\ref{apd:line:x:3} of the Algorithm~\ref{apd:alg} we get
	\begin{align*}
		\frac{1}{\eta_x}\sqn{x^{k+1} - x^*}
		&\leq
		\left(\frac{1}{\eta_x} - \alpha_x - \beta_x\mu_{yx}^2\right)\sqn{x^{k} - x^*}
		+
		\left(\beta_xL_{xy}^2 + \alpha_x-\frac{1}{\eta_x}\right)
		\sqn{x^{k+1} - x^k}
		+
		2\beta_xL_y\bg_g(y_g^k,y^*)
		\\&+
		\alpha_x\sqn{x_g^k - x^*}
		-
		\frac{2}{\sigma_x}\<\nabla f(x_g^k) - \nabla f(x^*),x_f^{k+1}- x_g^k>
		+
		\frac{2(1-\tau_x)}{\tau_x}\<\nabla f(x_g^k) - \nabla f(x^*),x_f^{k}- x_g^k>
		\\&-
		2\bg_f(x_g^k,x^*) - \mu_x\sqn{x_g^k -x^*}
		-
		2\<\mA^\top (y_m^k - y^*),x^{k+1} - x^*>
		\\&=
		\left(\frac{1}{\eta_x} - \alpha_x - \beta_x\mu_{yx}^2\right)\sqn{x^{k} - x^*}
		+
		\left(\beta_xL_{xy}^2 + \alpha_x-\frac{1}{\eta_x}\right)
		\sqn{x^{k+1} - x^k}
		+
		(\alpha_x-\mu_x)\sqn{x_g^k - x^*}
		\\&+
		2\beta_xL_y\bg_g(y_g^k,y^*)
		-
		2\bg_f(x_g^k,x^*) 
		-
		\frac{2}{\sigma_x}\<\nabla f(x_g^k) - \nabla f(x^*),x_f^{k+1}- x_g^k>
		\\&+
		\frac{2(1-\tau_x)}{\tau_x}\<\nabla f(x_g^k) - \nabla f(x^*),x_f^{k}- x_g^k>
		-
		2\<\mA^\top (y_m^k - y^*),x^{k+1} - x^*>.
	\end{align*}
	Using convexity of $\bg_f(x,x^*)$ with respect to $x$, which follows from the convexity of $f$, we get
	\begin{align*}
		\frac{1}{\eta_x}\sqn{x^{k+1} - x^*}
		&\leq
		\left(\frac{1}{\eta_x} - \alpha_x - \beta_x\mu_{yx}^2\right)\sqn{x^{k} - x^*}
		+
		\left(\beta_xL_{xy}^2 + \alpha_x-\frac{1}{\eta_x}\right)
		\sqn{x^{k+1} - x^k}
		+
		(\alpha_x-\mu_x)\sqn{x_g^k - x^*}
		\\&+
		2\beta_xL_y\bg_g(y_g^k,y^*)
		-
		2\bg_f(x_g^k,x^*) 
		-
		\frac{2}{\sigma_x}\<\nabla f(x_g^k) - \nabla f(x^*),x_f^{k+1}- x_g^k>
		\\&+
		\frac{2(1-\tau_x)}{\tau_x}\left(\bg_f(x_f^k,x^*) - \bg_f(x_g^k,x^*)\right)
		-
		2\<\mA^\top (y_m^k - y^*),x^{k+1} - x^*>.
	\end{align*}
	Using $L_x$-smoothness of $\bg_f(x,x^*)$ with respect to $x$, which follows from the $L_x$-smoothness of $f$, we get
	\begin{align*}
		\frac{1}{\eta_x}\sqn{x^{k+1} - x^*}
		&\leq
		\left(\frac{1}{\eta_x} - \alpha_x - \beta_x\mu_{yx}^2\right)\sqn{x^{k} - x^*}
		+
		\left(\beta_xL_{xy}^2 + \alpha_x-\frac{1}{\eta_x}\right)
		\sqn{x^{k+1} - x^k}
		+
		(\alpha_x-\mu_x)\sqn{x_g^k - x^*}
		\\&+
		2\beta_xL_y\bg_g(y_g^k,y^*)
		-
		2\bg_f(x_g^k,x^*) 
		-
		\frac{2}{\sigma_x}\left(\bg_f(x_f^{k+1},x^*) - \bg_f(x_g^k,x^*) - \frac{L_x}{2}\sqn{x_f^{k+1} - x_g^k}\right)
		\\&+
		\frac{2(1-\tau_x)}{\tau_x}\left(\bg_f(x_f^k,x^*) - \bg_f(x_g^k,x^*)\right)
		-
		2\<\mA^\top (y_m^k - y^*),x^{k+1} - x^*>.
	\end{align*}
	Using Line~\ref{apd:line:x:3} of the Algorithm~\ref{apd:alg} we get
	\begin{align*}
		\frac{1}{\eta_x}\sqn{x^{k+1} - x^*}
		&\leq
		\left(\frac{1}{\eta_x} - \alpha_x - \beta_x\mu_{yx}^2\right)\sqn{x^{k} - x^*}
		+
		\left(\beta_xL_{xy}^2 + \alpha_x-\frac{1}{\eta_x}\right)
		\sqn{x^{k+1} - x^k}
		+
		(\alpha_x-\mu_x)\sqn{x_g^k - x^*}
		\\&+
		2\beta_xL_y\bg_g(y_g^k,y^*)
		-
		2\bg_f(x_g^k,x^*) 
		-
		\frac{2}{\sigma_x}\left(\bg_f(x_f^{k+1},x^*) - \bg_f(x_g^k,x^*) - \frac{L_x\sigma_x^2}{2}\sqn{x^{k+1} - x^k}\right)
		\\&+
		\frac{2(1-\tau_x)}{\tau_x}\left(\bg_f(x_f^k,x^*) - \bg_f(x_g^k,x^*)\right)
		-
		2\<\mA^\top (y_m^k - y^*),x^{k+1} - x^*>
		\\&=
		\left(\frac{1}{\eta_x} - \alpha_x - \beta_x\mu_{yx}^2\right)\sqn{x^{k} - x^*}
		+
		\left(\beta_xL_{xy}^2 + \alpha_x + L_x\sigma_x -\frac{1}{\eta_x}\right)
		\sqn{x^{k+1} - x^k}
		\\&+
		(\alpha_x-\mu_x)\sqn{x_g^k - x^*}
		+
		2\beta_xL_y\bg_g(y_g^k,y^*)
		+
		\left(\frac{2}{\sigma_x} - \frac{2}{\tau_x}\right)\bg_f(x_g^k,x^*) 
		-
		\frac{2}{\sigma_x}\bg_f(x_f^{k+1},x^*) 
		\\&+
		\left(\frac{2}{\tau_x} - 2\right)\bg_f(x_f^k,x^*) 
		-
		2\<\mA^\top (y_m^k - y^*),x^{k+1} - x^*>.
	\end{align*}
	Using the definition of $\tau_x$, $\alpha_x$ and $\beta_x$ we get
	\begin{align*}
		\frac{1}{\eta_x}\sqn{x^{k+1} - x^*}
		&\leq
		\left(\frac{1}{\eta_x} - \mu_x - \beta_x\mu_{yx}^2\right)\sqn{x^{k} - x^*}
		+
		\left(\mu_x + L_x\sigma_x -\frac{1}{2\eta_x}\right)
		\sqn{x^{k+1} - x^k}
		\\&
		+
		\bg_g(y_g^k,y^*)
		-
		\bg_f(x_g^k,x^*) 
		-
		\frac{2}{\sigma_x}\bg_f(x_f^{k+1},x^*) 
		+
		\left(\frac{2}{\sigma_x} - 1\right)\bg_f(x_f^k,x^*) 
		\\&-
		2\<\mA^\top (y_m^k - y^*),x^{k+1} - x^*>.
	\end{align*}
\end{proof}

\begin{lemma}
	Let $\tau_y$ be defined as
	\begin{equation}
		\tau_y = (\sigma_y^{-1} + 1/2)^{-1}.
	\end{equation}
	Let $\alpha_y$ be defined as
	\begin{equation}
		\alpha_y = \mu_y.
	\end{equation}
	Let $\beta_y$ be defined as
	\begin{equation}
		\beta_y = \min\left\{\frac{1}{2L_x}, \frac{1}{2\eta_y L_{xy}^2}\right\}.
	\end{equation}
	Then, the following inequality holds:
	\begin{equation}\label{apd:eq:y}
		\begin{split}
			\frac{1}{\eta_y}\sqn{y^{k+1} - y^*}
			&\leq
			\left(\frac{1}{\eta_y} - \mu_y - \beta_y\mu_{xy}^2\right)\sqn{y^{k} - y^*}
			+
			\left(\mu_y + L_y\sigma_y -\frac{1}{2\eta_y}\right)
			\sqn{y^{k+1} - y^k}
			\\&
			+
			\bg_f(x_g^k,x^*)
			-
			\bg_g(y_g^k,y^*) 
			-
			\frac{2}{\sigma_y}\bg_g(y_f^{k+1},y^*) 
			+
			\left(\frac{2}{\sigma_y} - 1\right)\bg_g(y_f^k,y^*) 
			\\&+
			2\<\mA(x^{k+1} - x^*),y^{k+1} - y^*>.
		\end{split}
	\end{equation}
\end{lemma}
\begin{proof}
	The proof is similar to the proof of the previous lemma.
\end{proof}

\begin{lemma}
	Let $\eta_x$ be defined as
	\begin{equation}
		\eta_x = \min\left\{\frac{1}{4(\mu_x + L_x\sigma_x)},\frac{\delta}{4L_{xy}}\right\},
	\end{equation}
	and let $\eta_y$ be defined as
	\begin{equation}
		\eta_y = \min\left\{\frac{1}{4(\mu_y + L_y\sigma_y)},\frac{1}{4L_{xy}\delta}\right\},
	\end{equation}
	where $\delta > 0$ is a parameter.
	Let $\theta$ be defined as
	\begin{equation}\label{apd:theta}
		\theta = \theta(\delta,\sigma_x,\sigma_y) = 1 - \max\left\{\rho_a(\delta,\sigma_x,\sigma_y),\rho_b(\delta,\sigma_x,\sigma_y),\rho_c(\delta,\sigma_x,\sigma_y),\rho_d(\delta,\sigma_x,\sigma_y)\right\},
	\end{equation}
	where $\rho_a(\delta,\sigma_x,\sigma_y),\rho_b(\delta,\sigma_x,\sigma_y),\rho_c(\delta,\sigma_x,\sigma_y),\rho_d(\delta,\sigma_x,\sigma_y)$ are defined as 
	\begin{align}
		\rho_a(\delta,\sigma_x,\sigma_y) &= \left[\max\left\{\frac{4(\mu_x + L_x\sigma_x)}{\mu_x},\frac{2}{\sigma_x},\frac{4(\mu_y + L_y\sigma_y)}{\mu_y},\frac{2}{\sigma_y},\frac{4L_{xy}}{\mu_x\delta},\frac{4L_{xy}\delta}{\mu_y}\right\}\right]^{-1},\\
		\rho_b(\delta,\sigma_x,\sigma_y) &= \left[\max\left\{
		\frac{4(\mu_x + L_x\sigma_x)}{\mu_x},
		\frac{2}{\sigma_x},
		\frac{8L_x(\mu_y + L_y\sigma_y)}{\mu_{xy}^2},
		\frac{2}{\sigma_y},
		\frac{2L_{xy}^2}{\mu_{xy}^2},
		\frac{8L_xL_{xy}\delta}{\mu_{xy}^2},
		\frac{4L_{xy}}{\mu_x\delta}
		\right\}\right]^{-1},\\
		\rho_c(\delta,\sigma_x,\sigma_y) &= \left[\max\left\{
		\frac{4(\mu_y + L_y\sigma_y)}{\mu_y},
		\frac{2}{\sigma_y},
		\frac{8L_y(\mu_x + L_x\sigma_x)}{\mu_{yx}^2},
		\frac{2}{\sigma_x},
		\frac{2L_{xy}^2}{\mu_{yx}^2},
		\frac{8L_yL_{xy}}{\mu_{yx}^2\delta},
		\frac{4L_{xy}\delta}{\mu_y}
		\right\}\right]^{-1},\\
		\rho_d(\delta,\sigma_x,\sigma_y) &= \left[\max\left\{
		\frac{8L_y(\mu_x + L_x\sigma_x)}{\mu_{yx}^2},
		\frac{2}{\sigma_x},
		\frac{8L_x(\mu_y + L_y\sigma_y)}{\mu_{xy}^2},
		\frac{2}{\sigma_y},
		\frac{8L_yL_{xy}}{\delta\mu_{yx}^2},
		\frac{8L_xL_{xy}\delta}{\mu_{xy}^2},
		\frac{2L_{xy}^2}{\mu_{yx}^2},
		\frac{2L_{xy}^2}{\mu_{xy}^2}
		\right\}\right]^{-1}.
	\end{align}
	Let $\Psi^k$ be the following Lyapunov function:
	\begin{equation}
		\begin{split}
			\Psi^k &= 
			\frac{1}{\eta_x}\sqn{x^{k} - x^*}
			+
			\frac{1}{\eta_y}\sqn{y^{k} - y^*}
			+
			\frac{2}{\sigma_x}\bg_f(x_f^k,x^*) 
			+
			\frac{2}{\sigma_y}\bg_g(y_f^k,y^*)
			\\&+
			\frac{1}{4\eta_y}\sqn{y^k - y^{k-1}}
			-
			2\<y^k -y^{k-1},\mA(x^{k} - x^*)>.
		\end{split}
	\end{equation}
	Then, the following inequalities hold
	\begin{equation}\label{apd:eq:1}
		\Psi^k \geq \frac{3}{4\eta_x}\sqn{x^{k} - x^*}
		+
		\frac{1}{\eta_y}\sqn{y^{k} - y^*},
	\end{equation}
	\begin{equation}\label{apd:eq:2}
		\Psi^{k+1} \leq \theta\Psi^k.
	\end{equation}
\end{lemma}
\begin{proof}
	After adding up \eqref{apd:eq:x} and \eqref{apd:eq:y} we get
	\begin{align*}
		\mathrm{(LHS)}
		&\leq
		\left(\frac{1}{\eta_x} - \mu_x - \beta_x\mu_{yx}^2\right)\sqn{x^{k} - x^*}
		+
		\left(\frac{1}{\eta_y} - \mu_y - \beta_y\mu_{xy}^2\right)\sqn{y^{k} - y^*}
		\\&+
		\left(\mu_x + L_x\sigma_x -\frac{1}{2\eta_x}\right)
		\sqn{x^{k+1} - x^k}
		+
		\left(\mu_y + L_y\sigma_y -\frac{1}{2\eta_y}\right)
		\sqn{y^{k+1} - y^k}
		\\&+		
		\left(\frac{2}{\sigma_x} - 1\right)\bg_f(x_f^k,x^*) 
		+
		\left(\frac{2}{\sigma_y} - 1\right)\bg_g(y_f^k,y^*) 
		+
		2\<y^{k+1} - y_m^k,\mA(x^{k+1} - x^*)>.
	\end{align*}
	where $\mathrm{(LHS)}$ is given as
	\begin{align*}
		\mathrm{(LHS)}&= \frac{1}{\eta_x}\sqn{x^{k+1} - x^*} + \frac{1}{\eta_y}\sqn{y^{k+1} - y^*} + \frac{2}{\sigma_x}\bg_f(x_f^{k+1},x^*) +
		\frac{2}{\sigma_y}\bg_g(y_f^{k+1},y^*) .
	\end{align*}
	Using Line~\ref{apd:line:y:m} of the Algorithm~\ref{apd:alg} and Assumption~\ref{ass:A} we get
	\begin{align*}
		\mathrm{(LHS)}
		&\leq
		\left(\frac{1}{\eta_x} - \mu_x - \beta_x\mu_{yx}^2\right)\sqn{x^{k} - x^*}
		+
		\left(\frac{1}{\eta_y} - \mu_y - \beta_y\mu_{xy}^2\right)\sqn{y^{k} - y^*}
		\\&+
		\left(\mu_x + L_x\sigma_x -\frac{1}{2\eta_x}\right)
		\sqn{x^{k+1} - x^k}
		+
		\left(\mu_y + L_y\sigma_y -\frac{1}{2\eta_y}\right)
		\sqn{y^{k+1} - y^k}
		\\&+		
		\left(\frac{2}{\sigma_x} - 1\right)\bg_f(x_f^k,x^*) 
		+
		\left(\frac{2}{\sigma_y} - 1\right)\bg_g(y_f^k,y^*) 
		\\&+
		2\<y^{k+1} -y^k,\mA(x^{k+1} - x^*)>
		-
		2\theta\<y^k -y^{k-1},\mA(x^{k+1} - x^*)>
		\\&\leq
		\left(\frac{1}{\eta_x} - \mu_x - \beta_x\mu_{yx}^2\right)\sqn{x^{k} - x^*}
		+
		\left(\frac{1}{\eta_y} - \mu_y - \beta_y\mu_{xy}^2\right)\sqn{y^{k} - y^*}
		\\&+
		\left(\mu_x + L_x\sigma_x -\frac{1}{2\eta_x}\right)
		\sqn{x^{k+1} - x^k}
		+
		\left(\mu_y + L_y\sigma_y -\frac{1}{2\eta_y}\right)
		\sqn{y^{k+1} - y^k}
		\\&+		
		\left(\frac{2}{\sigma_x} - 1\right)\bg_f(x_f^k,x^*) 
		+
		\left(\frac{2}{\sigma_y} - 1\right)\bg_g(y_f^k,y^*) 
		\\&+
		2\<y^{k+1} -y^k,\mA(x^{k+1} - x^*)>
		-
		2\theta\<y^k -y^{k-1},\mA(x^{k} - x^*)>
		+
		2\theta L_{xy}\norm{y^k - y^{k-1}}\norm{x^{k+1} - x^k}.
	\end{align*}
	Using the definition of $\eta_x$ and $\eta_y$ and the fact that $\theta< 1$ we get
	\begin{align*}
		\mathrm{(LHS)}
		&\leq
		\left(\frac{1}{\eta_x} - \mu_x - \beta_x\mu_{yx}^2\right)\sqn{x^{k} - x^*}
		+
		\left(\frac{1}{\eta_y} - \mu_y - \beta_y\mu_{xy}^2\right)\sqn{y^{k} - y^*}
		\\&-
		\frac{1}{4\eta_x}
		\sqn{x^{k+1} - x^k}
		-
		\frac{1}{4\eta_y}
		\sqn{y^{k+1} - y^k}
		+		
		\left(\frac{2}{\sigma_x} - 1\right)\bg_f(x_f^k,x^*) 
		+
		\left(\frac{2}{\sigma_y} - 1\right)\bg_g(y_f^k,y^*) 
		\\&+
		2\<y^{k+1} -y^k,\mA(x^{k+1} - x^*)>
		-
		2\theta\<y^k -y^{k-1},\mA(x^{k} - x^*)>
		+
		\frac{\theta }{2\sqrt{\eta_x\eta_y}}\norm{y^k - y^{k-1}}\norm{x^{k+1} - x^k}
		\\&\leq
		\left(\frac{1}{\eta_x} - \mu_x - \beta_x\mu_{yx}^2\right)\sqn{x^{k} - x^*}
		+
		\left(\frac{1}{\eta_y} - \mu_y - \beta_y\mu_{xy}^2\right)\sqn{y^{k} - y^*}
		\\&-
		\frac{1}{4\eta_x}
		\sqn{x^{k+1} - x^k}
		-
		\frac{1}{4\eta_y}
		\sqn{y^{k+1} - y^k}
		+		
		\left(\frac{2}{\sigma_x} - 1\right)\bg_f(x_f^k,x^*) 
		+
		\left(\frac{2}{\sigma_y} - 1\right)\bg_g(y_f^k,y^*) 
		\\&+
		2\<y^{k+1} -y^k,\mA(x^{k+1} - x^*)>
		-
		2\theta\<y^k -y^{k-1},\mA(x^{k} - x^*)>
		+
		\frac{\theta}{4\eta_x}\sqn{x^{k+1} - x^k}
		+
		\frac{\theta}{4\eta_y}\sqn{y^k - y^{k-1}}
		\\&\leq
		\left(\frac{1}{\eta_x} - \mu_x - \beta_x\mu_{yx}^2\right)\sqn{x^{k} - x^*}
		+
		\left(\frac{1}{\eta_y} - \mu_y - \beta_y\mu_{xy}^2\right)\sqn{y^{k} - y^*}
		\\&+
		\frac{\theta}{4\eta_y}\sqn{y^k - y^{k-1}}
		-
		\frac{1}{4\eta_y}
		\sqn{y^{k+1} - y^k}
		+		
		\left(\frac{2}{\sigma_x} - 1\right)\bg_f(x_f^k,x^*) 
		+
		\left(\frac{2}{\sigma_y} - 1\right)\bg_g(y_f^k,y^*) 
		\\&+
		2\<y^{k+1} -y^k,\mA(x^{k+1} - x^*)>
		-
		2\theta\<y^k -y^{k-1},\mA(x^{k} - x^*)>.
	\end{align*}
	Using the definition of $\beta_x$ and $\beta_y$ we get
	\begin{align*}
		\mathrm{(LHS)}
		&\leq
		\left(1 - \eta_x\mu_x - \min\left\{\frac{\eta_x\mu_{yx}^2}{2L_y},\frac{\mu_{yx}^2}{2L_{xy}^2}\right\}\right)\frac{1}{\eta_x}\sqn{x^{k} - x^*}
		+
		\left(1 - \eta_y\mu_y - \min\left\{\frac{\eta_y\mu_{xy}^2}{2L_x},\frac{\mu_{xy}^2}{2L_{xy}^2}\right\}\right)\frac{1}{\eta_y}\sqn{y^{k} - y^*}
		\\&+
		\frac{\theta}{4\eta_y}\sqn{y^k - y^{k-1}}
		-
		\frac{1}{4\eta_y}
		\sqn{y^{k+1} - y^k}
		+		
		\left(\frac{2}{\sigma_x} - 1\right)\bg_f(x_f^k,x^*) 
		+
		\left(\frac{2}{\sigma_y} - 1\right)\bg_g(y_f^k,y^*) 
		\\&+
		2\<y^{k+1} -y^k,\mA(x^{k+1} - x^*)>
		-
		2\theta\<y^k -y^{k-1},\mA(x^{k} - x^*)>
		\\&\leq
		\left(1 - \max\left\{\eta_x\mu_x, \min\left\{\frac{\eta_x\mu_{yx}^2}{2L_y},\frac{\mu_{yx}^2}{2L_{xy}^2}\right\}\right\}\right)\frac{1}{\eta_x}\sqn{x^{k} - x^*}
		\\&+
		\left(1 - \max\left\{\eta_y\mu_y, \min\left\{\frac{\eta_y\mu_{xy}^2}{2L_x},\frac{\mu_{xy}^2}{2L_{xy}^2}\right\}\right\}\right)\frac{1}{\eta_y}\sqn{y^{k} - y^*}
		\\&+
		\frac{\theta}{4\eta_y}\sqn{y^k - y^{k-1}}
		-
		\frac{1}{4\eta_y}
		\sqn{y^{k+1} - y^k}
		+		
		\left(\frac{2}{\sigma_x} - 1\right)\bg_f(x_f^k,x^*) 
		+
		\left(\frac{2}{\sigma_y} - 1\right)\bg_g(y_f^k,y^*) 
		\\&+
		2\<y^{k+1} -y^k,\mA(x^{k+1} - x^*)>
		-
		2\theta\<y^k -y^{k-1},\mA(x^{k} - x^*)>.
	\end{align*}
	Using the definition of $\theta$ we get
	\begin{align*}
		\mathrm{(LHS)}
		&\leq
		\theta\left(
		\frac{1}{\eta_x}\sqn{x^{k} - x^*}
		+
		\frac{1}{\eta_y}\sqn{y^{k} - y^*}
		+
		\frac{1}{4\eta_y}\sqn{y^k - y^{k-1}}
		-
		2\<y^k -y^{k-1},\mA(x^{k} - x^*)>
		\right)
		\\&+
		\theta\left(
		\frac{2}{\sigma_x}\bg_f(x_f^k,x^*) 
		+
		\frac{2}{\sigma_y}\bg_g(y_f^k,y^*) 
		\right)	
		-
		\frac{1}{4\eta_y}
		\sqn{y^{k+1} - y^k}
		+
		2\<y^{k+1} -y^k,\mA(x^{k+1} - x^*)>.
	\end{align*}
	After rearranging and using the definition of $\Psi^k$ we get
	\begin{equation*}
		\Psi^{k+1}\leq \theta\Psi^k.
	\end{equation*}
	Finally, using the definition of $\Psi^k$, $\eta_x$ and $\eta_y$ we get
	\begin{align*}
		\Psi^k
		&\geq
		\frac{1}{\eta_x}\sqn{x^{k} - x^*}
		+
		\frac{1}{\eta_y}\sqn{y^{k} - y^*}
		+
		\frac{1}{4\eta_y}\sqn{y^k - y^{k-1}}
		-
		2\<y^k -y^{k-1},\mA(x^{k} - x^*)>
		\\&\geq
		\frac{1}{\eta_x}\sqn{x^{k} - x^*}
		+
		\frac{1}{\eta_y}\sqn{y^{k} - y^*}
		+
		\frac{1}{4\eta_y}\sqn{y^k - y^{k-1}}
		-
		2L_{xy}\norm{y^k -y^{k-1}}\norm{x^{k} - x^*}
		\\&\geq
		\frac{1}{\eta_x}\sqn{x^{k} - x^*}
		+
		\frac{1}{\eta_y}\sqn{y^{k} - y^*}
		+
		\frac{1}{4\eta_y}\sqn{y^k - y^{k-1}}
		-
		\frac{1}{2\sqrt{\eta_x\eta_y}}\norm{y^k -y^{k-1}}\norm{x^{k} - x^*}
		\\&\geq
		\frac{1}{\eta_x}\sqn{x^{k} - x^*}
		+
		\frac{1}{\eta_y}\sqn{y^{k} - y^*}
		+
		\frac{1}{4\eta_y}\sqn{y^k - y^{k-1}}
		-
		\frac{1}{4\eta_x}\sqn{x^{k} - x^*}
		-
		\frac{1}{4\eta_y}\sqn{y^k - y^{k-1}}
		\\&=
		\frac{3}{4\eta_x}\sqn{x^{k} - x^*}
		+
		\frac{1}{\eta_y}\sqn{y^{k} - y^*}.
	\end{align*}
\end{proof}

\begin{proof}[Proof of Theorem~\ref{apd:thm}]
	From \eqref{apd:eq:1} and \eqref{apd:eq:2} we can conclude that
	\begin{equation*}
		\frac{3}{4\eta_x}\sqn{x^{k} - x^*}
		+
		\frac{1}{\eta_y}\sqn{y^{k} - y^*} \leq \theta^k \Psi^0.
	\end{equation*}
	This implies the following inequality
	\begin{equation*}
		\max\left\{\sqn{x^{k} - x^*},\sqn{y^{k} - x^*}\right\}\leq\theta^k \Psi^0\max\left\{4\eta_x/3,\eta_y\right\}.
	\end{equation*}
	Hence, we can conclude that 
	\begin{equation*}
		\max\left\{\sqn{x^{k} - x^*},\sqn{y^{k} - x^*}\right\} \leq \epsilon,
	\end{equation*}
	as long as the number of iterations $k$ satisfies
	\begin{equation*}
		k\geq\frac{1}{1-\theta}\log\frac{C}{\epsilon},
	\end{equation*}
	where $C = \Psi^0\max\left\{4\eta_x/3,\eta_y\right\}$, which does not depend on $\epsilon$.
	From \eqref{apd:theta} we obtain
	\begin{align*}
		\frac{1}{1-\theta}&=\min\left\{\frac{1}{\rho_a(\delta,\sigma_x,\sigma_y)},\frac{1}{\rho_b(\delta,\sigma_x,\sigma_y)},\frac{1}{\rho_c(\delta,\sigma_x,\sigma_y)},\frac{1}{\rho_d(\delta,\sigma_x,\sigma_y)}\right\}.
	\end{align*}
	We can now try to approximately optimize parameters $\delta>0$ and $\sigma_x,\sigma_y \in (0,1]$ to obtain the smallest possible values of $\rho_a(\delta,\sigma_x,\sigma_y)^{-1},\rho_b(\delta,\sigma_x,\sigma_y)^{-1},\rho_c(\delta,\sigma_x,\sigma_y)^{-1},\rho_d(\delta,\sigma_x,\sigma_y)^{-1}$. This can be done in a closed form and the result is the following:
	\begin{align*}
		\frac{1}{\rho_a} &\leq
		4 + 4\max\left\{\sqrt{\frac{L_x}{\mu_x}}, \sqrt{\frac{L_y}{\mu_y}},\frac{L_{xy}}{\sqrt{\mu_x\mu_y}}\right\}
		\text{ for } \delta = \sqrt{\frac{\mu_y}{\mu_x}}, \sigma_x = \sqrt{\frac{\mu_x}{2L_x}},\sigma_y = \sqrt{\frac{\mu_x}{2L_x}},\\
		\frac{1}{\rho_b} & \leq
		4+8\max\left\{
		\frac{\sqrt{L_xL_y}}{\mu_{xy}},
		\frac{L_{xy}}{\mu_{xy}}\sqrt{\frac{L_x}{\mu_x}},
		\frac{L_{xy}^2}{\mu_{xy}^2}
		\right\}
		\text{ for } \delta = \sqrt{\frac{\mu_{xy}^2}{2\mu_xL_x}}, \sigma_x = \sqrt{\frac{\mu_x}{2L_x}},\sigma_y =\min\left\{1,\sqrt{\frac{\mu_{xy}^2}{4L_xL_y}}\right\},\\
		\frac{1}{\rho_c} & \leq
		4+8\max\left\{
		\frac{\sqrt{L_xL_y}}{\mu_{yx}},
		\frac{L_{xy}}{\mu_{yx}}\sqrt{\frac{L_y}{\mu_y}},
		\frac{L_{xy}^2}{\mu_{yx}^2}
		\right\}
		\text{ for } \delta = \sqrt{\frac{2\mu_yL_y}{\mu_{yx}^2}},\sigma_x =\min\left\{1,\sqrt{\frac{\mu_{yx}^2}{4L_xL_y}}\right\},\sigma_y = \sqrt{\frac{\mu_y}{2L_y}},\\
		\frac{1}{\rho_d} &\leq 2+8\max\left\{
		\frac{\sqrt{L_xL_y}L_{xy}}{\mu_{xy}\mu_{yx}},
		\frac{L_{xy}^2}{\mu_{yx}^2},
		\frac{L_{xy}^2}{\mu_{xy}^2}
		\right\}
		\text{ for } \delta = \frac{\mu_{xy}}{\mu_{yx}}\sqrt{\frac{L_y}{L_x}}, \sigma_x = \min\left\{1,\sqrt{\frac{\mu_{yx}^2}{4L_xL_y}}\right\},\sigma_y =\min\left\{1,\sqrt{\frac{\mu_{xy}^2}{4L_xL_y}}\right\}.
	\end{align*}
	Note, that we set $\mu_y = 0$ in the bound for $\rho_b^{-1}$, $\mu_x= 0$ in the bound for $\rho_c^{-1}$ and $\mu_x = \mu_y = 0$ in the bound for $\rho_d^{-1}$. This is a valid move, because any convex function is $0$-strongly convex by the definition of strong convexity.
\end{proof}

\newpage

\section{Proof of Theorem~\ref{gda:thm}} \label{sec:AppB}

\begin{lemma}
	Problem~\eqref{eq:main2} has a unique solution $(x^*, y^*)$.
\end{lemma}
\begin{proof}
	Consider operator $T\colon \R^{d_x}\times \R^{d_y} \rightarrow \R^{d_x}\times \R^{d_y}$ defined as $T \colon (x,y) \mapsto (x - t_x \nabla_x F(x,y),y - t_y \nabla_y F(x,y)) $ for some fixed $t_x,t_y > 0$. It is obvious that $(x,y)$ is a fixed point of operator $T$ if and only if $(x,y)$ is a solution to problem~\eqref{eq:main2}.
	If one can show that this operator is contractive, then it has a unique fixed point due to Banach fixed-point theorem.  The proof of the fact that $T$ is contractive is similar to the proof of the rest of \Cref{gda:thm}.
\end{proof}

\begin{lemma}
	Let $\eta_x$ be defined as
	\begin{equation}
		\eta_x = \min\left\{\frac{1}{8L_x}, \frac{\delta}{4L_{xy}}\right\},
	\end{equation}
	and let $\eta_y$ be defined as
	\begin{equation}
		\eta_y = \min\left\{\frac{1}{8L_y}, \frac{1}{4\delta L_{xy}}\right\},
	\end{equation}
	where $\delta>0$ is a parameter.
	Let $\theta$ be defined as
	\begin{equation}\label{gda:theta}
		\theta = \theta(\delta) = 1 - \max\left\{\rho_a(\delta),\rho_b(\delta),\rho_c(\delta),\rho_d(\delta)\right\},
	\end{equation}
	where $\rho_a(\delta),\rho_b(\delta),\rho_c(\delta),\rho_d(\delta)$ are defined as
	\begin{align}
		\frac{1}{\rho_a(\delta)} &= \max\left\{\frac{8L_x}{\mu_x},\frac{8L_y}{\mu_y},\frac{4L_{xy}}{\delta\mu_x},\frac{4L_{xy}\delta}{\mu_y}\right\},\\
		\frac{1}{\rho_b(\delta)} &= \max\left\{\frac{8L_x}{\mu_x},\frac{512L_xL_y}{\mu_{xy}^2},\frac{4L_{xy}}{\delta\mu_x},\frac{256L_xL_{xy}\delta}{\mu_{xy}^2},\frac{256L_yL_{xy}}{\mu_{xy}^2\delta},\frac{128L_{xy}^2}{\mu_{xy}^2}\right\},\\
		\frac{1}{\rho_c(\delta)} &=
		\max\left\{\frac{8L_y}{\mu_y},\frac{512L_xL_y}{\mu_{yx}^2},\frac{4L_{xy}\delta}{\mu_y},\frac{256L_xL_{xy}\delta}{\mu_{yx}^2},\frac{256L_yL_{xy}}{\mu_{yx}^2\delta},\frac{128L_{xy}^2}{\mu_{yx}^2}\right\}
		,\\
		\frac{1}{\rho_d(\delta)} &= \max\left\{\frac{512L_xL_y}{\min\{\mu_{xy}^2,\mu_{yx}^2\}},\frac{256L_xL_{xy}\delta}{\min\{\mu_{xy}^2,\mu_{yx}^2\}},\frac{256L_yL_{xy}}{\min\{\mu_{xy}^2,\mu_{yx}^2\}\delta},\frac{128L_{xy}^2}{\min\{\mu_{xy}^2,\mu_{yx}^2\}}\right\}.
	\end{align}
	Let $\Psi^k$ be the following Lyapunov function:
	\begin{equation}
		\Psi^k = \frac{1}{\eta_x}\sqn{x^k - x^*}
		+
		\frac{1}{\eta_y}\sqn{y^k - y^*} 
		-
		2\<\nabla_x F(x^{k-1},y^k) - \nabla_x F(x^{k-1},y^{k-1}), x^k  - x^* >
		+
		\frac{5}{16\eta_y}\sqn{y^k - y^{k-1}}.
	\end{equation}
	Then, the following inequalities hold
	\begin{equation}\label{gda:eq:1}
		\Psi^k \geq \frac{3}{4\eta_x}\sqn{x^{k} - x^*}
		+
		\frac{1}{\eta_y}\sqn{y^{k} - y^*},
	\end{equation}
	\begin{equation}\label{gda:eq:2}
		\Psi^{k+1} \leq \theta\Psi^k.
	\end{equation}
	
\end{lemma}

\begin{proof}
	Using Line~\ref{gda:line:x} of the Algorithm~\ref{gda:alg} we get.
	\begin{align*}
		\frac{1}{\eta_x}\sqn{x^{k+1} - x^*}
		&=
		\frac{1}{\eta_x}\sqn{x^k - x^*} + \frac{2}{\eta_x}\<x^{k+1} - x^k,x^{k+1} - x^*> - \frac{1}{\eta_x}\sqn{x^{k+1} - x^k}
		\\&=
		\frac{1}{\eta_x}\sqn{x^k - x^*} -  \frac{1}{\eta_x}\sqn{x^{k+1} - x^k}
		\\&-2\<\nabla_x F(x^k,y^k) + \theta(\nabla_x F(x^{k-1},y^k) - \nabla_x F(x^{k-1},y^{k-1})),x^{k+1} - x^*> 
		\\&=
		\frac{1}{\eta_x}\sqn{x^k - x^*} -   \frac{1}{\eta_x}\sqn{x^{k+1} - x^k} - 2\<\nabla_x F(x^k,y^{k+1}), x^{k+1} - x^k + x^{k} - x^* >
		\\&+
		2\<\nabla_x F(x^k,y^{k+1}) - \nabla_x F(x^k,y^k), x^{k+1}  - x^* > - 2\theta\<\nabla_x F(x^{k-1},y^k) - \nabla_x F(x^{k-1},y^{k-1}), x^{k+1}  - x^* >.
	\end{align*}
	Using the Assumption~\ref{ass:F1} we get
	\begin{align*}
		\frac{1}{\eta_x}\sqn{x^{k+1} - x^*}
		&\leq
		\left(\frac{1}{\eta_x} - \mu_x\right)\sqn{x^k - x^*}  +\left(L_x -   \frac{1}{\eta_x}\right)\sqn{x^{k+1} - x^k} - 2(F(x^{k+1},y^{k+1}) - F(x^*,y^{k+1}))	
		\\&+
		2\<\nabla_x F(x^k,y^{k+1}) - \nabla_x F(x^k,y^k), x^{k+1}  - x^* > - 2\theta\<\nabla_x F(x^{k-1},y^k) - \nabla_x F(x^{k-1},y^{k-1}), x^{k+1}  - x^* >.
	\end{align*}
	Using the Assumption~\ref{ass:F2} we get
	\begin{align*}
		\frac{1}{\eta_x}\sqn{x^{k+1} - x^*}
		&\leq
		\left(\frac{1}{\eta_x} - \mu_x\right)\sqn{x^k - x^*}  +\left(L_x -   \frac{1}{\eta_x}\right)\sqn{x^{k+1} - x^k} - 2(F(x^{k+1},y^{k+1}) - F(x^*,y^{k+1}))	
		\\&+
		2\<\nabla_x F(x^k,y^{k+1}) - \nabla_x F(x^k,y^k), x^{k+1}  - x^* > - 2\theta\<\nabla_x F(x^{k-1},y^k) - \nabla_x F(x^{k-1},y^{k-1}), x^k  - x^* >
		\\&+
		2L_{xy}\theta\norm{x^{k+1} - x^k}\norm{y^k - y^{k-1}}.
	\end{align*}
	Similarly, we can obtain the following  upper-bound on $\frac{1}{\eta_y}\sqn{y^{k+1} - y^*}$:
	\begin{equation*}
		\frac{1}{\eta_y}\sqn{y^{k+1} - y^*}
		\leq
		\left(\frac{1}{\eta_y} - \mu_y\right)\sqn{y^k - y^*}  +\left(L_y-   \frac{1}{\eta_y}\right)\sqn{y^{k+1} - y^k} + 2(F(x^{k+1},y^{k+1}) - F(x^{k+1},y^*))	.
	\end{equation*}
	Summing up the upper-bounds on $\frac{1}{\eta_x}\sqn{x^{k+1} - x^*}$ and $\frac{1}{\eta_y}\sqn{y^{k+1} - y^*}$ gives
	\begin{align*}
		\text{(LHS)}
		&\leq
		\left(\frac{1}{\eta_x} - \mu_x\right)\sqn{x^k - x^*}  +\left(L_x -   \frac{1}{\eta_x}\right)\sqn{x^{k+1} - x^k}
		\\&+
		\left(\frac{1}{\eta_y} - \mu_y\right)\sqn{y^k - y^*}  +\left(L_y-   \frac{1}{\eta_y}\right)\sqn{y^{k+1} - y^k}
		\\&+
		2L_{xy}\theta\norm{x^{k+1} - x^k}\norm{y^k - y^{k-1}}
		-
		2\theta\<\nabla_x F(x^{k-1},y^k) - \nabla_x F(x^{k-1},y^{k-1}), x^k  - x^* >
		\\&+
		2(F(x^*,y^{k+1}) - F(x^{k+1},y^*)),
	\end{align*}
	where (LHS) is defined as
	\begin{equation*}
		\text{(LHS)} = \frac{1}{\eta_x}\sqn{x^{k+1} - x^*} + \frac{1}{\eta_y}\sqn{y^{k+1} - y^*} - 2\<\nabla_x F(x^k,y^{k+1}) - \nabla_x F(x^k,y^k), x^{k+1}  - x^* >.
	\end{equation*}
	The Assumption~\ref{ass:F1} states, that function $F(x,y)$ is $L_x$-smooth in $x$ and $L_y$-smooth in $y$. Hence, using the optimality conditions \eqref{eq:opt2} we get
	\begin{align*}
		\text{(LHS)}
		&\leq
		\left(\frac{1}{\eta_x} - \mu_x\right)\sqn{x^k - x^*}  +\left(L_x -   \frac{1}{\eta_x}\right)\sqn{x^{k+1} - x^k}
		\\&+
		\left(\frac{1}{\eta_y} - \mu_y\right)\sqn{y^k - y^*}  +\left(L_y-   \frac{1}{\eta_y}\right)\sqn{y^{k+1} - y^k}
		\\&+
		2L_{xy}\theta\norm{x^{k+1} - x^k}\norm{y^k - y^{k-1}}
		-
		2\theta\<\nabla_x F(x^{k-1},y^k) - \nabla_x F(x^{k-1},y^{k-1}), x^k  - x^* >
		\\&-
		2(F(x^{k+1},y^*) - F(x^*,y^*))
		-
		2(F(x^*,y^*)-F(x^*,y^{k+1}))
		\\&\leq
		\left(\frac{1}{\eta_x} - \mu_x\right)\sqn{x^k - x^*}  +\left(L_x -   \frac{1}{\eta_x}\right)\sqn{x^{k+1} - x^k}
		\\&+
		\left(\frac{1}{\eta_y} - \mu_y\right)\sqn{y^k - y^*}  +\left(L_y-   \frac{1}{\eta_y}\right)\sqn{y^{k+1} - y^k}
		\\&+
		2L_{xy}\theta\norm{x^{k+1} - x^k}\norm{y^k - y^{k-1}}
		-
		2\theta\<\nabla_x F(x^{k-1},y^k) - \nabla_x F(x^{k-1},y^{k-1}), x^k  - x^* >
		\\&-
		\frac{\delta_x}{L_x}\sqn{\nabla_x F(x^{k+1}, y^*) }
		-
		\frac{\delta_y}{L_y}\sqn{\nabla_y F(x^*, y^{k+1})},
	\end{align*}
	where $\delta_x,\delta_y \in (0,1]$ are some parameters, that will be defined later.
	Using the Assumption~\ref{ass:F3} we get
	\begin{align*}
		\text{(LHS)}
		&\leq
		\left(\frac{1}{\eta_x} - \mu_x\right)\sqn{x^k - x^*}  +\left(L_x -   \frac{1}{\eta_x}\right)\sqn{x^{k+1} - x^k}
		\\&+
		\left(\frac{1}{\eta_y} - \mu_y\right)\sqn{y^k - y^*}  +\left(L_y-   \frac{1}{\eta_y}\right)\sqn{y^{k+1} - y^k}
		\\&+
		2L_{xy}\theta\norm{x^{k+1} - x^k}\norm{y^k - y^{k-1}}
		-
		2\theta\<\nabla_x F(x^{k-1},y^k) - \nabla_x F(x^{k-1},y^{k-1}), x^k  - x^* >
		\\&-
		\frac{\delta_x}{2L_x}\sqn{\nabla_x F(x^{k+1}, y^*)  - \nabla_x F(x^{k+1}, y^k) }
		+
		\frac{\delta_x}{L_x}\sqn{\nabla_x F(x^{k+1}, y^k) }
		\\&-
		\frac{\delta_y}{2L_y}\sqn{\nabla_y F(x^*, y^{k+1})-\nabla_y F(x^k, y^{k+1})}
		+
		\frac{\delta_y}{L_y}\sqn{\nabla_y F(x^k, y^{k+1})}
		\\&\leq
		\left(\frac{1}{\eta_x} - \mu_x - \frac{\delta_y\mu_{yx}^2}{2L_y}\right)\sqn{x^k - x^*}  +\left(L_x -   \frac{1}{\eta_x}\right)\sqn{x^{k+1} - x^k}
		\\&+
		\left(\frac{1}{\eta_y} - \mu_y - \frac{\delta_x\mu_{xy}^2}{2L_x}\right)\sqn{y^k - y^*}  +\left(L_y-   \frac{1}{\eta_y}\right)\sqn{y^{k+1} - y^k}
		\\&+
		2L_{xy}\theta\norm{x^{k+1} - x^k}\norm{y^k - y^{k-1}}
		-
		2\theta\<\nabla_x F(x^{k-1},y^k) - \nabla_x F(x^{k-1},y^{k-1}), x^k  - x^* >
		\\&+
		\frac{\delta_x}{L_x}\sqn{\nabla_x F(x^{k+1}, y^{k}) }
		+
		\frac{\delta_y}{L_y}\sqn{\nabla_y F(x^k, y^{k+1})}
	\end{align*}
	Using Lines~\ref{gda:line:x} and~\ref{gda:line:y} of the Algorithm~\ref{gda:alg}and the Lipschitzness property of $\nabla_x F(x,y)$ and $\nabla_y F(x,y)$ we get
	\begin{align*}
		\text{(LHS)}
		&\leq
		\left(\frac{1}{\eta_x} - \mu_x - \frac{\delta_y\mu_{yx}^2}{2L_y}\right)\sqn{x^k - x^*}  +\left(L_x -   \frac{1}{\eta_x}\right)\sqn{x^{k+1} - x^k}
		\\&+
		\left(\frac{1}{\eta_y} - \mu_y - \frac{\delta_x\mu_{xy}^2}{2L_x}\right)\sqn{y^k - y^*}  +\left(L_y-   \frac{1}{\eta_y}\right)\sqn{y^{k+1} - y^k}
		\\&+
		2L_{xy}\theta\norm{x^{k+1} - x^k}\norm{y^k - y^{k-1}}
		-
		2\theta\<\nabla_x F(x^{k-1},y^k) - \nabla_x F(x^{k-1},y^{k-1}), x^k  - x^* >
		\\&+
		\frac{2\delta_x}{L_x}\sqn{\nabla_x F(x^{k+1}, y^{k})  - \nabla_x F(x^k,y^k) - \theta(\nabla_x F(x^{k-1},y^k) - \nabla_x F(x^{k-1},y^{k-1}) )}
		+
		\frac{2\delta_x}{L_x\eta_x^2}\sqn{x^{k+1} - x^k}
		\\&+
		\frac{2\delta_y}{L_y}\sqn{\nabla_y F(x^k, y^{k+1}) - \nabla_y F(x^{k+1}, y^k)} + \frac{2\delta_y}{L_y\eta_y^2}\sqn{y^{k+1} - y^k}
		\\&\leq
		\left(\frac{1}{\eta_x} - \mu_x - \frac{\delta_y\mu_{yx}^2}{2L_y}\right)\sqn{x^k - x^*}  +\left(L_x -   \frac{1}{\eta_x}\right)\sqn{x^{k+1} - x^k}
		\\&+
		\left(\frac{1}{\eta_y} - \mu_y - \frac{\delta_x\mu_{xy}^2}{2L_x}\right)\sqn{y^k - y^*}  +\left(L_y-   \frac{1}{\eta_y}\right)\sqn{y^{k+1} - y^k}
		\\&+
		2L_{xy}\theta\norm{x^{k+1} - x^k}\norm{y^k - y^{k-1}}
		-
		2\theta\<\nabla_x F(x^{k-1},y^k) - \nabla_x F(x^{k-1},y^{k-1}), x^k  - x^* >
		\\&+
		4\delta_xL_x\sqn{x^{k+1} - x^k}
		+
		\frac{4\delta_x L_{xy}^2\theta^2}{L_x}\sqn{y^k - y^{k-1}}
		+
		\frac{2\delta_x}{L_x\eta_x^2}\sqn{x^{k+1} - x^k}
		\\&+
		4\delta_yL_y\sqn{y^{k+1} - y^k}
		+
		\frac{4\delta_yL_{xy}^2}{L_y}\sqn{x^{k+1} - x^k}
		+
		\frac{2\delta_y}{L_y\eta_y^2}\sqn{y^{k+1} - y^k}.
	\end{align*}
	Now, we set $\delta_x = \min\left\{1,c_x\eta_xL_x\right\}$, $\delta_y = \min\left\{1,c_y\eta_yL_y\right\}$, where $c_x,c_y>0$ will be defined later, and obtain
	\begin{align*}
		\text{(LHS)}
		&\leq
		\left(\frac{1}{\eta_x} - \mu_x - \frac{\delta_y\mu_{yx}^2}{2L_y}\right)\sqn{x^k - x^*}  +\left(L_x -   \frac{1}{\eta_x}\right)\sqn{x^{k+1} - x^k}
		\\&+
		\left(\frac{1}{\eta_y} - \mu_y - \frac{\delta_x\mu_{xy}^2}{2L_x}\right)\sqn{y^k - y^*}  +\left(L_y-   \frac{1}{\eta_y}\right)\sqn{y^{k+1} - y^k}
		\\&+
		2L_{xy}\theta\norm{x^{k+1} - x^k}\norm{y^k - y^{k-1}}
		-
		2\theta\<\nabla_x F(x^{k-1},y^k) - \nabla_x F(x^{k-1},y^{k-1}), x^k  - x^* >
		\\&+
		4c_x\eta_xL_x^2\sqn{x^{k+1} - x^k}
		+
		4c_x\eta_xL_{xy}^2\theta^2\sqn{y^k - y^{k-1}}
		+
		\frac{2c_x}{\eta_x}\sqn{x^{k+1} - x^k}
		\\&+
		4c_y\eta_yL_y^2\sqn{y^{k+1} - y^k}
		+
		4c_y\eta_yL_{xy}^2\sqn{x^{k+1} - x^k}
		+
		\frac{2c_y}{\eta_y}\sqn{y^{k+1} - y^k}.
	\end{align*}
	Using the definition of $\eta_x$ and $\eta_y$ we get
	\begin{align*}
		\text{(LHS)}
		&\leq
		\left(\frac{1}{\eta_x} - \mu_x - \frac{\delta_y\mu_{yx}^2}{2L_y}\right)\sqn{x^k - x^*}  +\left(L_x -   \frac{1}{\eta_x}\right)\sqn{x^{k+1} - x^k}
		\\&+
		\left(\frac{1}{\eta_y} - \mu_y - \frac{\delta_x\mu_{xy}^2}{2L_x}\right)\sqn{y^k - y^*}  +\left(L_y-   \frac{1}{\eta_y}\right)\sqn{y^{k+1} - y^k}
		\\&
		-
		2\theta\<\nabla_x F(x^{k-1},y^k) - \nabla_x F(x^{k-1},y^{k-1}), x^k  - x^* >
		\\&+
		4c_x\eta_xL_x^2\sqn{x^{k+1} - x^k}
		+
		\frac{(c_x+1)\theta^2}{4\eta_y}\sqn{y^k - y^{k-1}}
		+
		\frac{2c_x}{\eta_x}\sqn{x^{k+1} - x^k}
		\\&+
		4c_y\eta_yL_y^2\sqn{y^{k+1} - y^k}
		+
		\frac{c_y+1}{4\eta_x}\sqn{x^{k+1} - x^k}
		+
		\frac{2c_y}{\eta_y}\sqn{y^{k+1} - y^k}.
	\end{align*}
	Now, we choose $c_x = c_y = \frac{1}{4}$ and get
	\begin{align*}
		\text{(LHS)}
		&\leq
		\left(\frac{1}{\eta_x} - \mu_x - \frac{\delta_y\mu_{yx}^2}{2L_y}\right)\sqn{x^k - x^*}  +\left(L_x -   \frac{1}{\eta_x}\right)\sqn{x^{k+1} - x^k}
		\\&+
		\left(\frac{1}{\eta_y} - \mu_y - \frac{\delta_x\mu_{xy}^2}{2L_x}\right)\sqn{y^k - y^*}  +\left(L_y-   \frac{1}{\eta_y}\right)\sqn{y^{k+1} - y^k}
		\\&
		-
		2\theta\<\nabla_x F(x^{k-1},y^k) - \nabla_x F(x^{k-1},y^{k-1}), x^k  - x^* >
		\\&+
		\eta_xL_x^2\sqn{x^{k+1} - x^k}
		+
		\frac{5\theta^2}{16\eta_y}\sqn{y^k - y^{k-1}}
		+
		\frac{1}{2\eta_x}\sqn{x^{k+1} - x^k}
		\\&+
		\eta_yL_y^2\sqn{y^{k+1} - y^k}
		+
		\frac{5}{16\eta_x}\sqn{x^{k+1} - x^k}
		+
		\frac{1}{2\eta_y}\sqn{y^{k+1} - y^k}
		\\&=
		\left(\frac{1}{\eta_x} - \mu_x - \frac{\delta_y\mu_{yx}^2}{2L_y}\right)\sqn{x^k - x^*}
		+
		\frac{\eta_xL_x +\eta_x^2L_x^2- 3/16}{\eta_x}\sqn{x^{k+1} - x^k}
		\\&+
		\left(\frac{1}{\eta_y} - \mu_y - \frac{\delta_x\mu_{xy}^2}{2L_x}\right)\sqn{y^k - y^*} 
		+
		\frac{\eta_yL_y +\eta_y^2L_y^2- 3/16}{\eta_y}\sqn{y^{k+1} - y^k}
		\\&-
		2\theta\<\nabla_x F(x^{k-1},y^k) - \nabla_x F(x^{k-1},y^{k-1}), x^k  - x^* >
		+
		\frac{5\theta^2}{16\eta_y}\sqn{y^k - y^{k-1}} - \frac{5}{16\eta_y}\sqn{y^{k+1} - y^k}.
	\end{align*}
	Using the definition of $\eta_x$ and $\eta_y$ we get
	\begin{align*}
		\text{(LHS)}
		&\leq
		\left(\frac{1}{\eta_x} - \mu_x - \frac{\delta_y\mu_{yx}^2}{2L_y}\right)\sqn{x^k - x^*}
		+
		\left(\frac{1}{\eta_y} - \mu_y - \frac{\delta_x\mu_{xy}^2}{2L_x}\right)\sqn{y^k - y^*} 
		\\&-
		2\theta\<\nabla_x F(x^{k-1},y^k) - \nabla_x F(x^{k-1},y^{k-1}), x^k  - x^* >
		+
		\frac{5\theta^2}{16\eta_y}\sqn{y^k - y^{k-1}} - \frac{5}{16\eta_y}\sqn{y^{k+1} - y^k}.
	\end{align*}
	Using the definition of $\delta_x$ and $\delta_y$ we get
	\begin{align*}
		\text{(LHS)}
		&\leq
		\left(1 - \max\left\{\eta_x\mu_x,\min\left\{\frac{\eta_x\mu_{yx}^2}{2L_y},\frac{\eta_x\eta_y\mu_{yx}^2}{8}\right\}\right\}\right)\frac{1}{\eta_x}\sqn{x^k - x^*}
		\\&+
		\left(1 - \max\left\{\eta_y\mu_y,\min\left\{\frac{\eta_y\mu_{xy}^2}{2L_x},\frac{\eta_y\eta_x\mu_{xy}^2}{8}\right\}\right\}\right)\frac{1}{\eta_y}\sqn{y^k - y^*} 
		\\&-
		2\theta\<\nabla_x F(x^{k-1},y^k) - \nabla_x F(x^{k-1},y^{k-1}), x^k  - x^* >
		+
		\frac{5\theta^2}{16\eta_y}\sqn{y^k - y^{k-1}} - \frac{5}{16\eta_y}\sqn{y^{k+1} - y^k}.
	\end{align*}
	Using the definition of $\eta_x,\eta_y$ and $\theta$ we get
	\begin{align*}
		\text{(LHS)}
		&\leq
		\frac{\theta}{\eta_x}\sqn{x^k - x^*}
		+
		\frac{\theta}{\eta_y}\sqn{y^k - y^*} 
		-
		2\theta\<\nabla_x F(x^{k-1},y^k) - \nabla_x F(x^{k-1},y^{k-1}), x^k  - x^* >
		+
		\frac{5\theta}{16\eta_y}\sqn{y^k - y^{k-1}}
		\\&-
		\frac{5}{16\eta_y}\sqn{y^{k+1} - y^k}.
	\end{align*}
	After rearranging and using the definition of $\Psi^k$ we get
	\begin{equation*}
		\Psi^{k+1}\leq \theta\Psi^k.
	\end{equation*}
	Finally, using the definition of $\Psi^k$, $\eta_x$ and $\eta_y$ we get
	\begin{align*}
		\Psi^k
		&=
		\frac{1}{\eta_x}\sqn{x^k - x^*}
		+
		\frac{1}{\eta_y}\sqn{y^k - y^*} 
		-
		2\<\nabla_x F(x^{k-1},y^k) - \nabla_x F(x^{k-1},y^{k-1}), x^k  - x^* >
		+
		\frac{5}{16\eta_y}\sqn{y^k - y^{k-1}}
		\\&\geq
		\frac{1}{\eta_x}\sqn{x^k - x^*}
		+
		\frac{1}{\eta_y}\sqn{y^k - y^*} 
		-
		2L_{xy}\norm{y^k - y^{k-1}}\norm{x^k - x^*}
		+
		\frac{5}{16\eta_y}\sqn{y^k - y^{k-1}}
		\\&\geq
		\frac{1}{\eta_x}\sqn{x^k - x^*}
		+
		\frac{1}{\eta_y}\sqn{y^k - y^*} 
		-
		\frac{1}{4\eta_x}\sqn{x^k - x^*}
		-
		\frac{1}{4\eta_y}\sqn{y^k - y^{k-1}}
		+
		\frac{5}{16\eta_y}\sqn{y^k - y^{k-1}}
		\\&\geq
		\frac{3}{4\eta_x}\sqn{x^k - x^*}
		+
		\frac{1}{\eta_y}\sqn{y^k - y^*} .
	\end{align*}
\end{proof}

\begin{proof}[Proof of Theorem~\ref{gda:thm}]
	From \eqref{gda:eq:1} and \eqref{gda:eq:2} we can conclude that
	\begin{equation*}
		\frac{3}{4\eta_x}\sqn{x^{k} - x^*}
		+
		\frac{1}{\eta_y}\sqn{y^{k} - y^*} \leq \theta^k \Psi^0.
	\end{equation*}
	This implies the following inequality
	\begin{equation*}
		\max\left\{\sqn{x^{k} - x^*},\sqn{y^{k} - x^*}\right\}\leq\theta^k \Psi^0\max\left\{4\eta_x/3,\eta_y\right\}.
	\end{equation*}
	Hence, we can conclude that 
	\begin{equation*}
		\max\left\{\sqn{x^{k} - x^*},\sqn{y^{k} - x^*}\right\} \leq \epsilon,
	\end{equation*}
	as long as the number of iterations $k$ satisfies
	\begin{equation*}
		k\geq\frac{1}{1-\theta}\log\frac{C}{\epsilon},
	\end{equation*}
	where $C = \Psi^0\max\left\{4\eta_x/3,\eta_y\right\}$, which does not depend on $\epsilon$.
	From \eqref{gda:theta} we obtain
	\begin{align*}
		\frac{1}{1-\theta}&=\min\left\{\frac{1}{\rho_a(\delta)},\frac{1}{\rho_b(\delta)},\frac{1}{\rho_c(\delta)},\frac{1}{\rho_d(\delta)}\right\}.
	\end{align*}
	Now, we find the parameter $\delta$ to obtain the following upper bounds on $\rho_a(\delta),\rho_b(\delta),\rho_c(\delta),\rho_d(\delta)$:
	\begin{align}
		\frac{1}{\rho_a} &= \max\left\{
		\frac{8L_x}{\mu_x},
		\frac{8L_y}{\mu_y},
		\frac{4L_{xy}}{\sqrt{\mu_x\mu_y}}
		\right\}
		\text{ for } \delta = \sqrt{\frac{\mu_y}{\mu_x}},\\
		\frac{1}{\rho_b} &= \max\left\{
		\frac{8L_x}{\mu_x},
		\frac{512L_xL_y}{\mu_{xy}^2},
		\frac{128L_{xy}^2}{\mu_{xy}^2}
		\right\}
		\text{ for } \delta = \max\left\{\frac{\mu_{xy}}{8\sqrt{\mu_xL_x}},\sqrt{\frac{L_y}{L_x}}\right\},\\
		\frac{1}{\rho_c} &=
		\max\left\{
		\frac{8L_y}{\mu_y},
		\frac{512L_xL_y}{\mu_{yx}^2},
		\frac{128L_{xy}^2}{\mu_{yx}^2}
		\right\}
		\text{ for } \delta = \min\left\{\frac{8\sqrt{\mu_yL_y}}{\mu_{yx}},\sqrt{\frac{L_y}{L_x}}\right\},\\
		\frac{1}{\rho_d} &= \max\left\{
		\frac{512L_xL_y}{\mu_{xy}^2},
		\frac{512L_xL_y}{\mu_{yx}^2},
		\frac{128L_{xy}^2}{\mu_{xy}^2},
		\frac{128L_{xy}^2}{\mu_{yx}^2}\right\}
		\text{ for } \delta = \sqrt{\frac{L_y}{L_x}}.
	\end{align}
\end{proof}

\end{document}